\documentclass{amsart} %{aptpub}
\usepackage{amsmath,amsfonts,amssymb}
\usepackage{graphicx}
\usepackage{xcolor}

\newtheorem{thm}{Theorem}
\newtheorem{prop}[thm]{Proposition}
\newtheorem{lem}[thm]{Lemma}
\newtheorem{cor}[thm]{Corollary}

\newtheorem{rem}[thm]{Remark}
\numberwithin{equation}{section} \numberwithin{thm}{section}

\newcommand{\R}{{\mathbb R}}
\newcommand{\N}{{\mathbb N}}
\newcommand{\Z}{{\mathbb Z}}
\newcommand{\E}{{\mathbb E}}

\renewcommand{\P}{{\mathbb P}}
\newcommand{\eps}{\varepsilon}
\newcommand{\sS}{\mathcal{S}}
\newcommand{\sT}{\mathcal{T}}

\newcommand{\intr}{\mathsf{int}}
\newcommand{\bd}{\partial}
\newcommand{\var}{\mathsf{var}}
\newcommand{\cl}{\overline}
\newcommand{\ind}{\mathbf{1}}
\newcommand{\sR}{\mathcal R}%{\mathcal{Z}}
\newcommand{\oV}{\overline{V}}
\newcommand{\osV}{\overline{\mathcal V}}
\newcommand{\Ha}{\mathcal{H}}

\numberwithin{equation}{section}  % If you number theorems, etc. within sections,
\begin{document}

\title{Geometric functionals of fractal percolation}

\author{Michael A.\ Klatt}
\address{Department of Physics, Princeton University, Princeton, New Jersey 08544, USA, mklatt@princeton.edu}
\author{Steffen Winter}
\address{Karlsruhe Institute of Technology, Department of Mathematics, 76128 Karlsruhe, Germany, steffen.winter@kit.edu}
% insert title - use \\ if it requires more than one line.

%\authorone[Affiliation]{Author} % Affiliation is just the name of your university or institution
%
%\addressone{Full address} % Your postal address goes here.

\begin{abstract}
Fractal percolation exhibits a dramatic topological phase transition,
changing abruptly from a {dust-like set to} a system spanning cluster. %percolating
The transition points are unknown and difficult to estimate.
In many classical percolation models the percolation thresholds have been approximated well using additive geometric functionals, known as intrinsic volumes.
Motivated by the question whether a similar approach is possible for fractal models,
we introduce corresponding geometric functionals for the fractal percolation process $F$. They arise as limits of expected functionals of finite approximations of $F$.
We establish the existence of these limit functionals and obtain explicit
formulas for them as well as for their finite approximations.
\end{abstract}

 % insert keywords separated by a semicolon

           % insert the primary Maths Subject Classification number in the first bracket
           % and the secondary ams number(s) in the second bracket
           % e.g. \ams{60E20}{49G03;49F10}

\maketitle

\section{Introduction} \label{sec:intro}

Fractal percolation in $\R^d$ is a family of random subsets of the unit cube $J:=[0,1]^d\subset\R^d$ depending on two parameters $M\in\N_{\geq 2}$
and $p\in[0,1]$, which is informally defined as follows: In the first step divide $J$ into $M^d$ closed subcubes of side length $1/M$. Each of these subcubes is kept with probability $p$ and discarded with probability $1-p$ independently of all other subcubes.
Then this construction is {iterated.} Let $F_n$, $n\in\N$, denote the union of the subcubes kept in the $n$-th step. {They arise as follows.} Assuming that $F_{n-1}$ is already constructed, in the $n$-th step each cube in $F_{n-1}$ (of side length $1/M^{n-1}$) is  divided into $M^d$ subcubes (of side length $1/M^{n}$) and each of these subcubes is kept (and included in $F_n$) with probability $p$ independently of all other subcubes and of the previous steps.
This way one obtains a decreasing sequence $F_{0}:=J\supset F_{1}\supset F_{2}\supset \ldots$ of (possibly empty) random compact sets. The limit set
\begin{align}
  F:= \bigcap_{n\in\N_0} F_{n}
\end{align}
is known as \emph{fractal percolation} or \emph{Mandelbrot percolation}, see e.g.\
\cite{Mandelbrot74,CCD88}. It is well known that $F$ is almost surely empty if $p\leq 1/M^d$, i.e.\ if on average not more than one of the $M^d$ subcubes of any cube in the construction survives. For $p>1/M^d$, however, there is a positive probability (depending on $p,M$ and $d$) that $F\neq \emptyset$, and conditioned on $F$ being nonempty,  the Hausdorff dimension and equally the Minkowski dimension of $F$ are almost surely given by the number
\begin{align} \label{eq:dimF}
   \dim_H F = D:=\frac{\log(M^d p)}{\log(M)} =d-\frac{\log(1/p)}{\log(M)},
\end{align}
see e.g.\ \cite{CCD88}.
The sets $F$ are among the simplest examples of self-similar random sets as introduced in  \cite{Falconer86,Graf87,MW86}.

Beside many other properties, in particular their connectivity has been studied.
{Fractal percolation exhibits a dramatic topological phase transition
-- for all dimensions $d\geq 2$ and all $M\in\N_{\geq 2}$ -- when the
parameter $p$ increases from 0 to 1. There is a critical probability
$p_c=p_c(M,d)\in(0,1)$ such that, for $p<p_c$, the set $F$ is almost
surely totally disconnected (`\emph{dustlike}'), and, for $p\geq p_c$,
$F$ has connected components larger than one point with probability 1
(provided $F$ is not empty), see \cite{BroCam10}.
Remarkably, the phase transition is discontinuous, that is, the
probability of connected components larger than one point in $F$ is
strictly positive at $p_c$.}

{For $d=2$, Chayes, Chayes and Durrett \cite{CCD88} observed that for
$p\geq p_c$ there is even a positive probability that $F$
\emph{percolates}, meaning here that $F$ has a connected component which
intersects the left and the right boundary of $J$, that is
$\{0\}\times[0,1]^{d-1}$ and $\{1\}\times[0,1]^{d-1}$.
Note that also at $p_c$,
there is a positive probability that $F$ percolates, a simpler proof of
this fact was provided in \cite{DM90}.}

{In dimension $d\geq 3$, $F$ is known to percolate with positive
probability for all $p>p_c(M,d)$, provided $M$ is large enough, see
\cite{BroCam10}. This is conjectured to hold for all $M$.
An open question is whether there is a positive probability for
percolation at the corresponding threshold in dimensions $d\geq 3$.
We refer to \cite{BroCam08,BroCam10} for a more detailed
discussion of this and related issues.}

  \begin{figure}[t]
      \begin{center}
        \includegraphics[width=\textwidth]{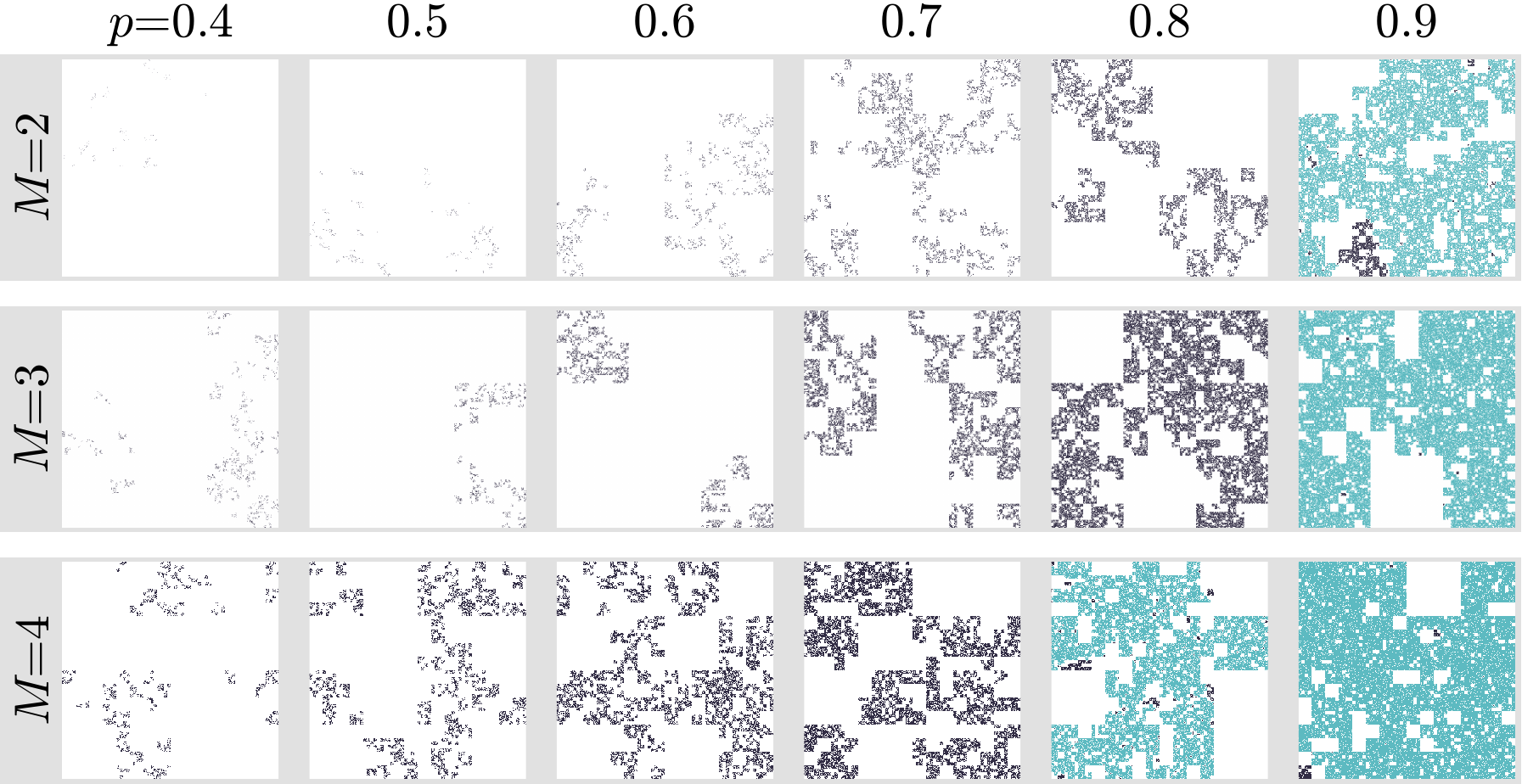}
      \end{center}
      \caption{\label{fig:realizations} Finite approximations of fractal percolation: realizations for different values of the survival probability $p$ and the linear number of subdivisions $M$. `Percolating' clusters that span the system in vertical and horizontal direction are {highlighted}.}
  \end{figure}

Like in many other percolation models, the exact values of $p_c(M,d)$
are not known.
In fact, for this model the situation is even worse than usual.
Classical techniques using finite size scaling apparently fail in this
fractal model, since a proper scaling regime is inaccessible with modern
hardware.
Some rigorous lower and upper bounds on $p_c(M,d)$ have been obtained in
particular for $d=2$, see Sec.~\ref{sec:percolation}, but they are not
tight.

Morphometric methods to estimate thresholds in percolation models have
been proposed in~\cite{MeckeWagner1991} and intensively studied in the
physics literature~\cite{MeckeSeyfried2002,NMW08,KSM17,RouCol16}, {see also the recent study in homological percolation \cite{BobSkr19} using topological data analysis.}
{These methods} are based on additive functionals from integral geometry, in
particular the Euler characteristic, and rely on the observation that in
many percolation models the expected Euler characteristic per site (as a
function of the model parameter $p$)--which can easily be computed
analytically in many models--has a zero close to the percolation
threshold of the model.
{In dimension 2, a heuristic explanation is based on the
  representation of the Euler characteristic for polyconvex sets as the
  difference between the number of components and the number of holes.
Below the threshold, there are many small connected components, which can
hardly contain any holes,
resulting in a positive Euler characteristic.
Above the threshold, few large clusters form with complex
shapes that contain many holes, resulting in a negative Euler
characteristic.}
{The argument can, in fact, be related to an exact derivation of the
threshold for certain self-matching lattices~\cite{NMW08}.}
Based on empirical evidence and heuristic arguments, the zeros of the Euler characteristic provide
for many classes of percolation models reasonable approximations and
putative bounds on the thresholds that capture their dependence on
system parameters like the degree of anisotropy~\cite{KSM17}.
{More precisely, across many models, the zero of the Euler characteristic
  has been observed to be a lower bound for the percolation threshold
  $p_c$ if $p_c<\frac 12$, and to be an upper bound for $p_c$ if
  $p_c>\frac 12$.}
  {For self-matching lattices with $p_c=\frac 12$ (like the triangular
  lattice), the zero of the Euler characteristic is also found to be
  $\frac 12$.
  This case is the only one up to now in which the relation between
  Euler characteristic and percolation thresholds has been proven
  rigorously using symmetry arguments, see e.g.~\cite{NMW08}.}

In analogy to these findings for discrete and continuum percolation models, we introduce and study here some corresponding geometric functionals for fractal percolation and ask the question whether one can use them to predict or at least approximate percolation thresholds. It is natural to expect that the dramatic phase transition (from dust to strong connectivity) in these fractal models should leave at least some trace in geometric functionals such as the Euler characteristic.
Due to the self-similarity of the model, there is even some hope that -- although percolation is a global property -- the thresholds can be predicted by local information alone.

Note that $F$ as well as the construction steps $F_{n}$ %for each $n\in\N$)
are random compact subsets of the unit cube $[0,1]^d$. Moreover, due to their construction, the sets $F_{n}$ are finite (random) unions of cubes (of side length $1/M^n$). Therefore, each $F_{n}$ is almost surely \emph{polyconvex}, i.e., a finite union of convex sets, and so intrinsic volumes $V_0(F_{n}), V_1(F_{n}),\ldots,
V_d(F_{n})$ (also known as Minkowski functionals) and even curvature
measures are well defined for $F_{n}$ almost surely.

\begin{figure}[t]
  \begin{center}
    \includegraphics[width=0.8\textwidth]{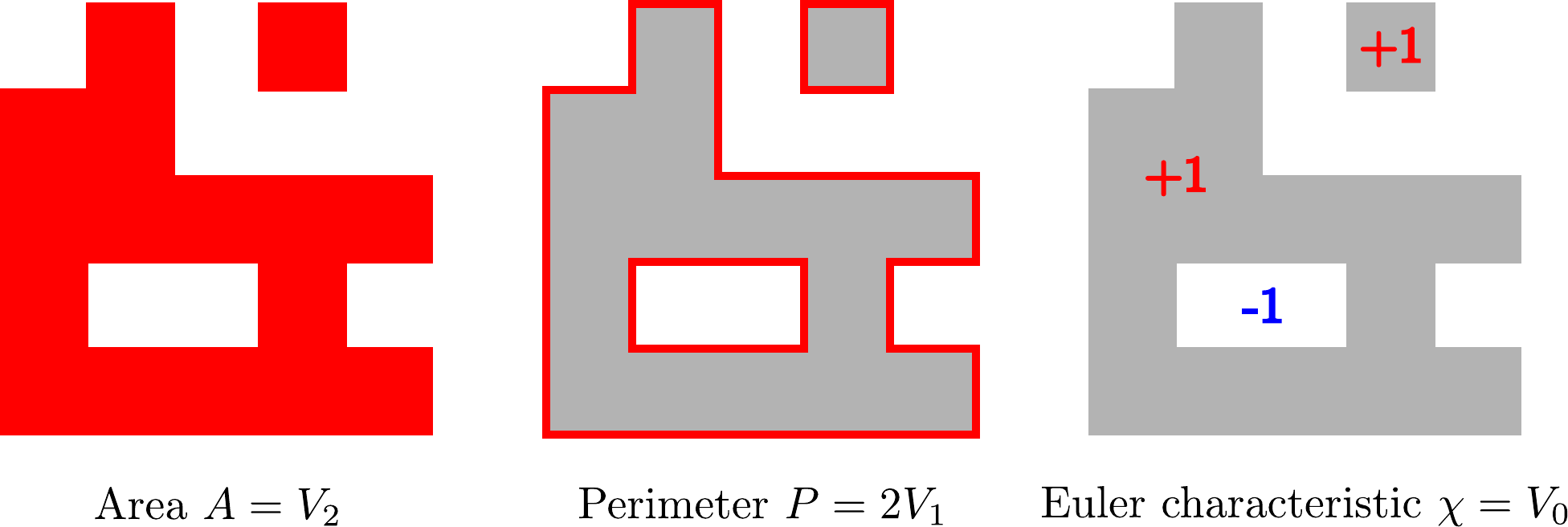}
  \end{center}
  \caption{\label{fig:Minkfuncts}
  {Illustration of the three intrinsic volumes in $\R^2$ -- area,
    boundary length and Euler characteristic (i.e., \#components $-$
  \#holes).}}
\end{figure}

{Intrinsic volumes form a (complete) system of geometric invariants, which are characterized by their properties. Among them are the volume $V_d$, the surface area $V_{d-1}$ and the Euler characteristic $V_0$, the remaining ones describe integrated curvature properties of the boundary.
For compact, convex sets $K\subset\R^d$, they are most easily introduced via the Steiner formula, which expresses the volume (Lebesgue measure) of the parallel sets $K_{\oplus \eps}:=\{x\in\R^d:\inf_{y\in K}||x-y||\leq \eps\}$, where $||\cdot||$ denotes the Euclidean norm, as a polynomial in $\eps$. The \emph{intrinsic volumes} $V_0(K),\ldots, V_d(K)$ of $K$ arise as the unique coefficients in this formula (up to normalization):
\begin{align*}
  V_d(K_{\oplus\eps})=\sum_{k=0}^d \kappa_{d-k} V_k(K) \eps^{d-k}.
\end{align*}
Here $\kappa_j$ denotes the volume of the unit ball in $\R^j$.
Intrinsic volumes are \emph{additive} functionals, i.e.\ %for any $k=0,\ldots,d$ and any
for compact, convex sets $K,L\subset\R^d$ the relation
\begin{align} \label{eq:additive}
   V_k(K)+V_k(L)=V_k(K\cup L)+V_k(K\cap L)
\end{align}
holds, provided $K\cup L$ is again convex.} They can be extended
additively to the \emph{convex ring} $\sR^d$, the family of all compact
polyconvex sets. Since $\sR^d$ is closed with respect to unions and
intersections, equation \eqref{eq:additive} holds for any $K,L\in\sR^d$.
{Among the further properties of intrinsic volumes are invariance
with respect to rigid motions  and homogeneity, meaning that $V_k$
satisfies $V_k(\lambda K)=\lambda^k V_k(K)$ for any $K\in\sR^d$ and any
$\lambda>0$, where $\lambda K:=\{\lambda x: x\in K\}$. We refer
to \cite[Ch.~4]{Schneider14} or \cite[Ch.~14.2]{SchneiderWeil08} for more details on intrinsic volumes.}

%\subsection*{Approximation of $F$ by the sequence $(F_{n})_n$.}
While for the sets $F_n$ intrinsic volumes are well defined, the limit set $F$ is a fractal and so these functionals are not directly defined. Since the $F_n$ approximate the limit set $F$, as $n\to \infty$, we are interested in the expectations $\E V_k(F_{n})$, $k\in\{0,\ldots,d\}$,
and in particular in their limiting behaviour as $n\to\infty$.
It turns out that some appropriate rescaling is necessary in order to see convergence, which is closely related to the Hausdorff (and Minkowski) dimension of the limit set $F$. Our first main result is a general formula which expresses these limits in terms of lower dimensional mutual intersections of certain parts of the construction steps $F_n$.
 Let $J_1,\ldots, J_{M^d}$ be the $M^d$ closed subcubes into which $[0,1]^d$ is divided in the first step of the construction of $F$.
Denote by $F^j_{n}$, $j=1,\ldots,M^d$, the union of all subcubes kept in
the $n$-th step, that are contained in $J_j$ ({see
  Figure~\ref{fig:Fnj} for an illustration and} \eqref{eq:Fj_ndef} for a formal definition).
\begin{figure}[t]
     \begin{center}
     \includegraphics[width=0.9\textwidth]{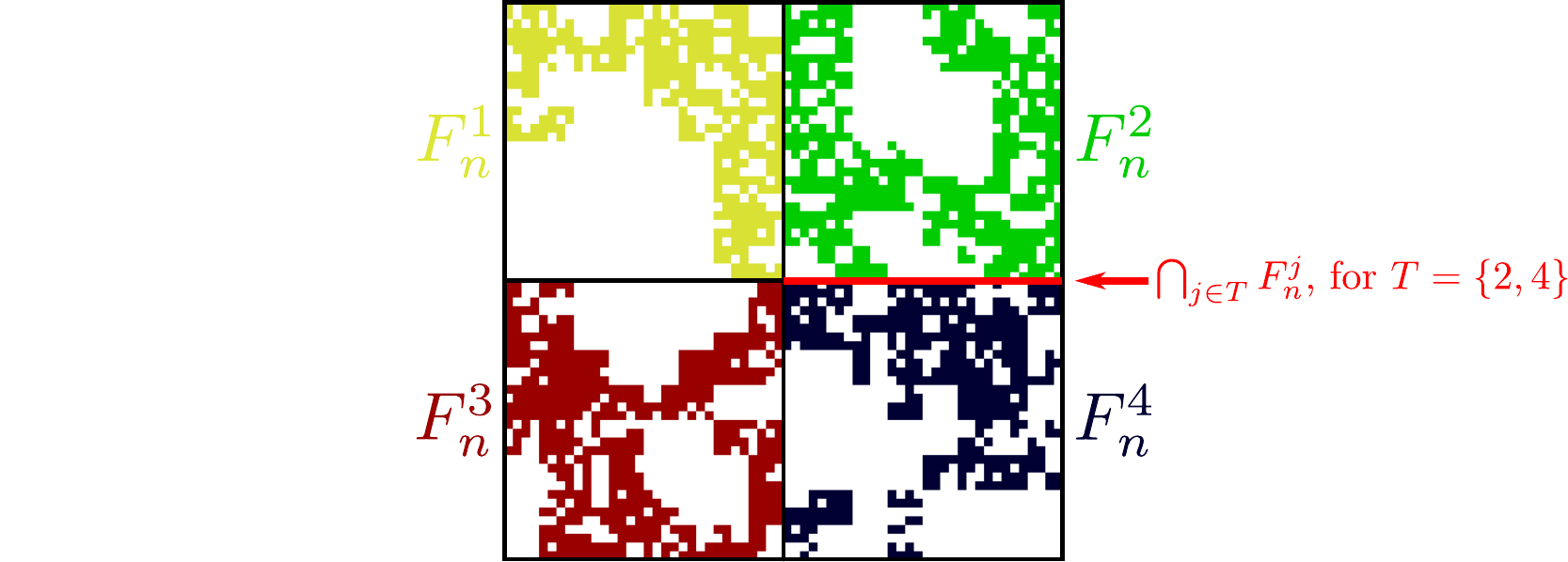}
     \end{center}
     \caption{ \label{fig:Fnj}{Illustration of the sets $F_n^j$
      as subsets of $F_n$ (for $M=2$ and $n=6$)
      and of an intersection $\bigcap_{j\in T} F_n^j$.
       The number of subsets $F_n^j$ is by definition always $M^2$.}
      %for $T=\{2,4\}$ is contained in the segment shown in red.
      }
\end{figure}
\begin{thm} \label{thm:Vk-limit-general}
Let $F$ %$=F^{d,M,p}$
be a fractal percolation on $[0,1]^d$ with parameters $M\in\N_{\geq 2}$ and {$p\in(0,1]$}. %(M^{-\min\{3,d\}},1]$. %$p\in(1/M^3,1]$.
Let $D$ be {as given by~\eqref{eq:dimF}} and let $r:=1/M$.
  Then, for each $k\in\{0,\ldots,d\}$, the limit
  $$
  \osV_k(F):=\lim_{n\to\infty} r^{n(D-k)} \E V_k(F_{n})
  $$
  exists and is given by the expression
  \begin{align} \label{eq:Vk-limit-general}
    q_{d,k} +\sum_{T\subset\{1,\ldots,M^d\},|T|\geq 2}(-1)^{|T|-1} \sum_{n=1}^\infty r^{n(D-k)} \E V_k(\bigcap_{j\in T} F^j_{n}),
  \end{align}
  where $q_{d,k}:=V_k([0,1]^d)$ is the $k$-th intrinsic volume of the unit cube in $\R^d$.
\end{thm}

We point out that all the intersections occurring in \eqref{eq:Vk-limit-general} consist of at least two of the cubes $F^j_{n}$ and are thus contained in some hyperplane. Hence, {in this formula} only sets appear which can be studied in a lower dimensional ambient space allowing to use fractal percolations in lower dimensional cubes for the computations. This makes the formula practically useful for explicit calculations as carried out in $\R^1$ and $\R^2$ below. Note also that many of the intersections are actually empty and that there are a lot of symmetries between the remaining ones.

{While Theorem~\ref{thm:Vk-limit-general} states the existence
of the limits $\osV_k(F)$ for all parameters $p\in(0,1]$, the limit
  set $F$ is empty almost surely for $p\leq M^{-d}$. So, the limit $\osV_k(F)$
is not a functional of the limit set $F$ only, but
it also depends on the chosen approximation sequence $F_n$.
}

In $\R^2$ (and similarly for $\R$, see Corollary~\ref{cor:Vk-dim1}) we use the formula in Theorem~\ref{thm:Vk-limit-general} to derive more explicit expressions for the limits $\osV_k(F)$.
\begin{thm}%\citeg{KW18}
\label{thm:Vk-limit-dim2}
Let $F$ be a fractal percolation in $[0,1]^2$ with parameters $M\in\N_{\geq 2}$ and {$p\in(0,1]$}.
  Then,
  \begin{align*}
    \osV_2(F) %=\lim_{n\to\infty} Z_2(n)
    &=%\lim_{n\to\infty} r^{n(D-2)} \E V_2(F_{n})=q_{2,2}=
    1, \qquad
    \osV_1(F) %=\lim_{n\to\infty} Z_1(n)
    =%\lim_{n\to\infty} r^{n(D-1)} \E V_1(F_{n})=q_{2,1}-2M(M-1)\sum_{n=1}^\infty  r^{n(D-1)} \E V_k(F_n^1\cap F_n^2)=
    %2 - 2M(M-1) \sum_{n=1}^\infty  r^{n(D-1)} \E V_1(F_n^1\cap F_n^2)
    \frac{2M(1-p)}{M-p} \qquad \text{ and } \\
    \osV_0(F)&=1-\frac{2p(M-1)^2}{M-p}\left(\frac 3{M-1}-\frac{4p}{M-p}+\frac{p^2}{M-p^2}\right)\\
        &\qquad +\frac{2p(M^2-1)}{M^2-p}-\frac{4p^2(M-1)^2}{(M-p)^2}+\frac{p^3(M-1)^2(M+p^2)}{(M-p^2)(M^2-p^3)}.
  \end{align*}
 \end{thm}

While $\osV_2(F)$ (the rescaled limit of the expected area) is constant and thus independent of $M$ and $p$, the functional $\osV_1(F)$ (the rescaled limit of the expected boundary lengths) is monotone decreasing in $p$ (for each fixed $M$). Most interesting is the limit $\osV_0(F)$ of the expected Euler characteristics of $F_{n}$.

Figure~\ref{fig:V0}\,(left) shows $\osV_0(F(p))$ as a function of the
survival probability $p$ for different $M$ (black curves).
The dotted vertical line indicates the threshold below which $F$ is
almost surely empty.
The coloured curves depict the analytic expressions for finite approximations of
the limit $\osV_0(F(p))$ by the rescaled functionals $p\mapsto r^{nD(p)}
\E V_0(F_{n}(p))$ for different $n$ (obtained in the proof of Theorem~\ref{thm:Vk-limit-dim2}).
Already for $n=12$ the curves are almost indistinguishable from
$n=\infty$, indicating a fast convergence, which is rigorously confirmed
below, see Remark~\ref{rem:speed}.
The formulae for finite approximations are compared to simulations, see Remark~\ref{rem:simu}.
The marks depict the arithmetic mean over 2500 to 75000 samples
(depending on $n$).
The error bars depict the standard error of the mean.
The simulation results are in excellent agreement with the analytic
curves.

The functionals $\osV_k(F)${,} which are based on the approximation of $F$ by the sequence $F_n$, provide a natural and intuitive first approach to quantify the geometry of fractal percolation $F$. One should however keep in mind that these limits most likely also depend on the approximation sequence. There are other natural sequences of sets which approximate $F$ well and which may even be better suited to capture certain aspects of the geometry of $F$. In particular, the parallel sets $F_{\oplus\eps}$, $\eps>0$ of $F$ are considered a good means of approximation, preserving many properties, and have been studied extensively also for (deterministic and random) self-similar sets, cf.\ e.g.\ \cite{W08,WZ13,Zaehle11}. Although the existence of the resulting limits (known as \emph{fractal curvatures}) has been established for random self-similar sets in \cite{Zaehle11}, parallel set approximation seems technically too difficult in order to derive explicit expressions for these limits even for the simplest examples. {We point out that the limit functionals $\osV_k(F)$ studied here are --- just like their relatives, the fractal curvatures -- closer in spirit to Minkowski content and dimension than to the Hausdorff dimension. They depend on (and are thus in principle capable of describing) how the studied set is embedded in the ambient space.} %-- one reason why we have started with the sequence $F_n$ above.

Note that in the current approach using the sets $F_n$ we consider closed cubes, meaning that two surviving subcubes in any finite approximation $F_n$ are connected even if they touch each other at a single corner. In the limit set $F$ such connections cannot survive (because it would require an infinite number of consecutive successes in a Bernoulli experiment with success probability $p$: the survival at each level $n$ of the two level-$n$ squares touching the corner).
Therefore, it might be advisable to seek for an approximation which avoids diagonal connections from the beginning. Such an approximation is provided by the closed complements of the $F_n$.
By connecting the subcubes in the complement, such non-surviving connections get disconnected already in the finite approximations $F_n$. More precisely, we study in Section ~\ref{sec:5} expectations $\E V_k(C_n)$, where $C_n:=\overline{[0,1]^d\setminus F_n}$ are the closed complements of the $F_n$ in the unit cube, and the limits
$$
   \osV^c_k(F):=\lim_{n\to\infty} r^{n(D-k)} \E V_k(C_n),
$$
with $D$ as in \eqref{eq:dimF}.
 We obtain for these limits a general formula (see Theorem~\ref{thm:Zk-limit-general2}), which is very similar to the one obtained in Theorem~\ref{thm:Vk-limit-general} for $\osV_k(F)$.
Again, for the case $d=2$, we have computed explicit expressions. Here we state only the formula for the Euler characteristic (i.e., the case $k=0$), the most interesting functional in connection with the percolative behaviour to be discussed in the next section (for the case $k=1$ see Proposition~\ref{prop:Vck-limit-dim2}).
 \begin{thm} \label{thm:Vck-limit-dim2}
Let $F$ be a fractal percolation in $[0,1]^2$ with parameters $M\in\N_{\geq 2}$ and $p\in(1/M^2,1]$.
  Then,
  \begin{align*}
       \osV^c_0(F)&= M^2(1-p)\frac{p^3+(M-1)p^2+(M-1)p-M}{(M^2-p^3)(M-p)}.
  \end{align*}
\end{thm}
\begin{figure}[t]
     \begin{center}
       \includegraphics[width=\textwidth]{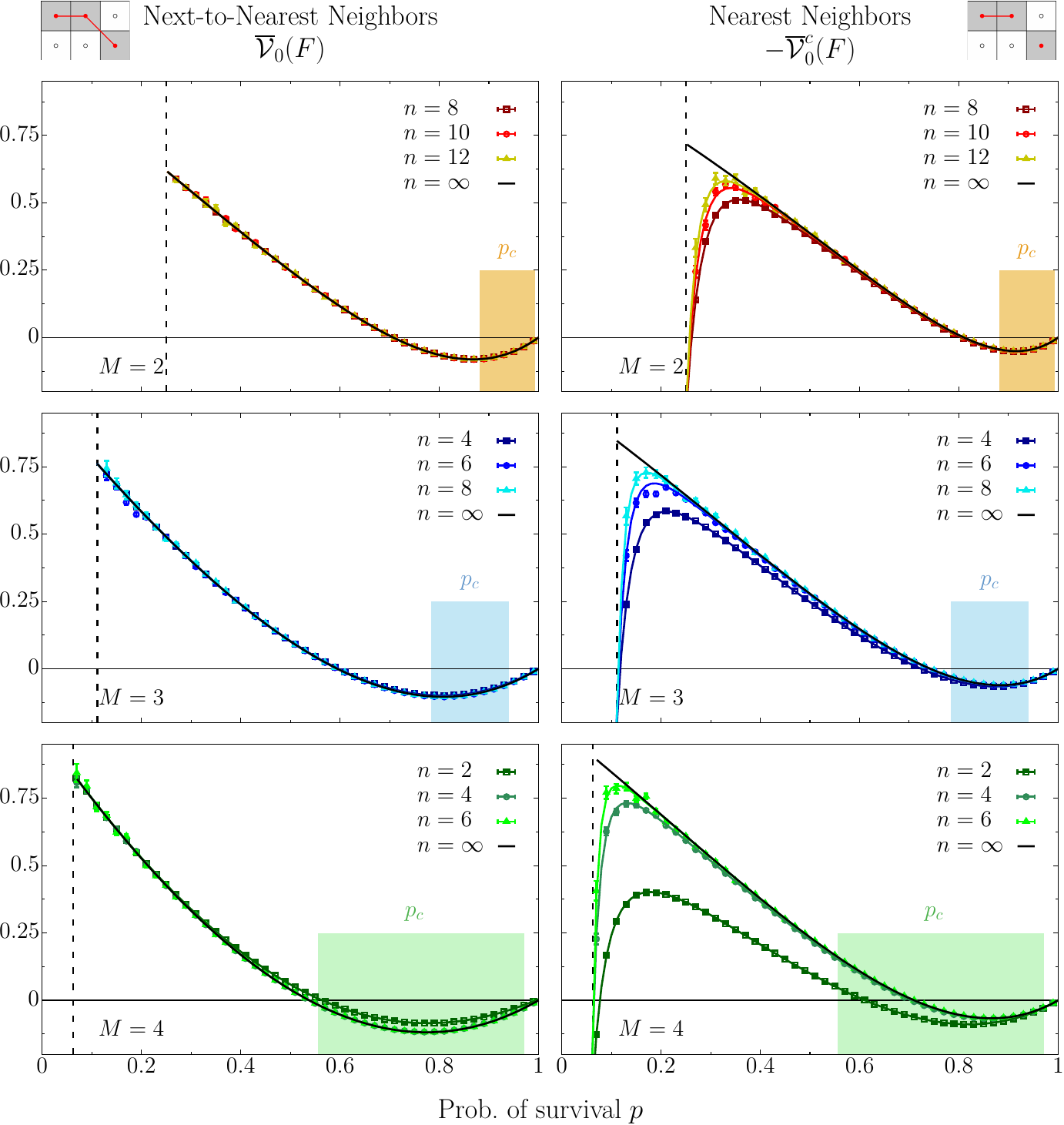}
     \end{center}
     \caption{ \label{fig:V0}
     Rescaled expected Euler characteristic of finite approximations $F_n$ (left) and their closed complements $C_n$ (right) as functions of the survival probability $p$ for $M=2$ (top), $M=3$ (center) and $M=4$ (bottom).
     Each plot compares finite approximations with increasing $n$ to the limit curve ($n=\infty$), that is, to $p\mapsto \osV_0(F(p))$ given in Theorem~\ref{thm:Vk-limit-dim2} (left) and $p\mapsto -\osV_0^c(F(p))$ given in Theorem~\ref{thm:Vck-limit-dim2}.
     The {shaded} areas indicate the rigorously known bounds on the percolation threshold, see~\eqref{eq:pc-bounds2}.}
\end{figure}
Note that in $\R^2$, $-V_0(C_n)$ is essentially the Euler characteristic of the set $F_n$ with all diagonal connections between cubes removed (up to some boundary effects along the boundary of $[0,1]^2$). Therefore, $-\osV^c_0(F)$ will be the functional of interest in the sequel in connection with the percolation properties of $F$.

More precisely, $1-V_0(C_n)+V_0(C_n\cap\partial[0,1]^2)$ corresponds to the Euler characteristic of the cell complex with vertex set given by the squares of $F_{n}$, edges between any two squares if they intersect in a common side and faces given by four edges forming a square. It can be shown that the effect of the last summand $V_0(C_n\cap\partial[0,1]^2)$ is asymptotically negligible. The approach corresponds to considering nearest neighbors in $\Z^2$ -- as opposed to taking also next-to-nearest neighbors into account, as done before.

\begin{figure}[t]
     \begin{center}
     \includegraphics[width=\textwidth]{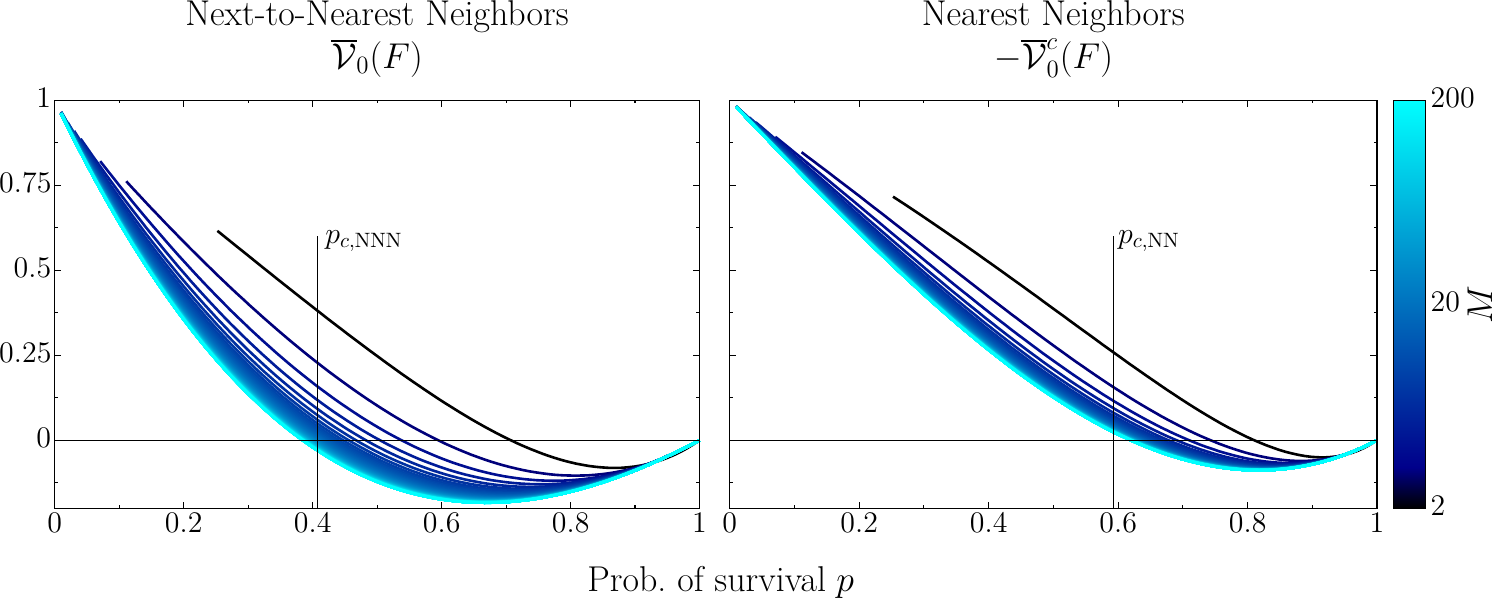}
     \end{center}
     \caption{ \label{fig:V1}
      For increasing values of the number of subdivisions $M$ (color coded), the rescaled expected Euler characteristic of fractal percolation (left) and its complement (right) are plotted as functions of $p$.
      The limiting curve (red) for $M\to\infty$ corresponds to the mean Euler characteristic per site (rescaled by the intensity) of site percolation on $\Z^2$ with eight or four neighbors, respectively.
     }
\end{figure}

\section{Relation with percolation thresholds} \label{sec:percolation}

We start by recalling some known results concerning the percolation thresholds $p_c=p_c(M)$ of fractal percolation in the plane. Already Chayes, Chayes and Durrett \cite{CCD88} established that, for any $M\in \N_{\geq 2}$,
\begin{align} \label{eq:pc-bounds}
 \sqrt{1/M}\leq p_c(M)\leq 0.9999.
\end{align}
In \cite{CC89} it is shown that the percolation threshold $p_{c,\text{NN}}$ of site percolation on the nearest neighbor (NN) graph on $\Z^2$ is a lower bound, i.e.\ $p_{c,\text{NN}}\leq p_c(M)$ for any $M\in\N_{\geq 2}$. Since $ 0.556\leq p_{c,\text{NN}}$, cf.~\cite{BE96}, this improves the above lower bound for any $M\geq 4$. Moreover, $\lim_{M\to\infty} p_c(M)=p_{c,\text{NN}}$, see \cite{CC89}.  It is believed that $p_c(M')\leq p_c(M)$ for $M'\geq M$ but this monotonicity is only established in special cases, e.g.\ if $M'=M^2$.
These bounds have been improved in \cite{White01,Don15} for some small $M$. The best known bounds for $M=2$ and $3$ are
\begin{align}
   \label{eq:pc-bounds2} 0.881\leq p_c(2)\leq 0.993 \quad \text{ and }\quad 0.784\leq p_c(3)\leq 0.940,
\end{align} respectively, and $p_c(4)\leq 0.972$, cf.~\cite{Don15}.

In view of the aforementioned observations in \cite{MeckeWagner1991, MeckeSeyfried2002, NMW08, KSM17}, that the zero of the expected Euler characteristic per site is close to the percolation thresholds in many percolation models, let us now discuss the connections between the limit functionals for $F$ introduced above and the connectivity properties in fractal percolation.

{\em $p_0$ is a lower bound for $p_c$.} Our first observation is that,
for any $M\in\N_{\geq 2}$, the function $p\mapsto \osV_0(F(p))$ has a
unique zero $p_0=p_0(M)$ in the open interval $(1/M^2,1)$, as suggested
by Figure \ref{fig:V0}\,(left). Moreover, $\osV_0(F(p))>0$ for $p<p_0$
and $\osV_0(F(p))<0$ for $p>p_0$.
{So far, this is in accordance with the observations in
classical percolation models.
Since $p_c>\frac12$, one would by analogy expect $p_0$ to be an upper
bound for $p_c$.
This is also what the naive heuristics suggests for our
  model:
below $p_c$ the limit set $F$ is totally disconnected and therefore, if
$\osV_0(F)$ is naively interpreted as the `Euler characteristic' of $F$
(i.e., as \#components $-$ \#holes), then it should be positive for all
$p<p_c$.}
By comparing $p_0$ with the known lower bounds for $p_c${, we find in contrast that $p_0(M)$ is not an upper bound but a lower bound for $p_c(M)$, i.e.}\
\begin{align*}
   p_0(M)\leq p_c(M), \quad \text{ for all } M\in\N_{\geq 2}.
\end{align*}
Indeed, for $M=2$ and $3$, $p_0(M)$ is below the lower bounds for $p_c(M)$ of Don, cf.\ \eqref{eq:pc-bounds2}, while for $M\geq 4$, $p_0(M)\leq 0.556$, which is the lower bound for the site percolation threshold $p_{c,\text{NN}}$ due to van den Berg and Ermakov \cite{BE96}.
      Although $p_0$ is a lower bound for $p_c$, unfortunately it is not very tight. In particular, it does not improve the known bounds.

\label{it:2} {\em The large-$M$ limit of $p_0(M)$.}
{In} analogy with $p_c$, for which $p_c(M)\to p_{c,\text{NN}}$ as $M\to\infty$, we observe that also the zeros $p_0(M)$ converge to a limit, as $M\to\infty$.  The (pointwise) limit of the functions $p\mapsto\osV_0(F(p))$, as $M\to \infty$, is the function $v$ given by
      $$
      v(p):=1-4p+4p^2-p^3, \qquad p\in(0,1],
      $$
      which is the red curve depicted in Figure~\ref{fig:V1} (left). It turns out that $v(p)$ coincides (up to a factor $p$) with the {\em mean Euler characteristic per site} $\oV_0(\Z^{2,\text{NNN}};p)$ of site percolation on the next-to-nearest neighbor (NNN) graph on $\Z^2$, cf.\ ~\cite{NMW08}. In particular, this implies for the zeros that
      \begin{align*}
        \lim_{M\to\infty} p_0(M)=p_{0,\text{NNN}},
      \end{align*}
      where $p_{0,\text{NNN}}=(3-\sqrt{5})/2$ is the unique zero of $v$ in $(0,1)$, i.e.\ of $p\mapsto \oV_0(\Z^{2,\text{NNN}};p)$.
      At first glance it might be surprising that a different site percolation model appears in the limit (NNN instead of NN, which showed up for the percolation thresholds). But this is consistent with the discussion before Theorem~\ref{thm:Vck-limit-dim2} -- there is too much connectivity in the approximation sets $F_n$. We will get back to this in a moment.

 {\em $p_{\min}$ -- a bound for $p_c$?} For any $M\in\N_{\geq 2}$, the
 function $p\mapsto \osV_0(F_p)$ has a unique minimum
 $p_{\min}=p_{\min}(M)$ in the open interval $(1/M^2,1)$, which lies
 always to the right of $p_0$ (i.e., potentially  closer to $p_c$). This
 is another natural candidate to bound the percolation threshold. {At
   $p_{\min}$ the geometry is extremal in the sense that, intuitively
   speaking, the `growth speed' of the number of holes equals exactly the `growth speed' of the number of clusters.}
      For $M=2$, $p_{\min}$ is clearly a lower bound for $p_c(2)$, but as $M\to\infty$, $p_{\min}(M)\to 2/3$, which is above $p_{c,\text{NN}}$. So, for large $M$, $p_{\min}(M)$ is clearly not a lower bound for $p_c(M)$. This implies that $p_{\min}(M)$ can neither be a general lower nor a general upper bound for the percolation thresholds. Interesting open questions are at which $M$, $p_c(M)$ and $p_{\min}(M)$ change their order and whether $p_{\min}(M)$ (which can be interpreted as the parameter for which the difference between number of holes and the number of connected components is maximal) is related in some way to the percolation transition.

As the discussion before Theorem~\ref{thm:Vck-limit-dim2} suggests, there might be approximation sequences for $F$ which better capture the percolative behaviour of $F$ and one candidate sequence are the modified sets $F_n$ with all diagonal connections between cubes removed, which we studied by looking at the closed complements $C_n:=\overline{J\setminus F_n}$.
Let us now discuss possible connections with percolation thresholds of the corresponding limit functionals $\osV^c_0(F)$.

{\em $p_1$ -- a lower bound for $p_c$?} Figure~\ref{fig:V0} (right) shows plots of the functions $p\mapsto-\osV_0^c(F(p))$ for different $M$ (the black curves labelled `$n=\infty$'), again accompanied by some finite approximations for different $n$.
In Figure~\ref{fig:V1} (right) there are  plots of the functions $p\mapsto-\osV_0^c(F(p))$ for all $M$ together with the limit curve as $M\to\infty$.  Each of these curves possesses again a unique zero $p_1=p_1(M)$ in $(1/M^2,1)$.
It is apparent from the plots in Figure~\ref{fig:V0}, that $p_1(M)$ is larger than $p_0(M)$ and thus potentially closer to the percolation threshold $p_c(M)$. At least for $M=2, 3$, $p_1(M)$ is a better lower bound for $p_c(M)$. But is this true in general? Unfortunately not, as will become clear from looking at large $M$.

\label{it:2-2} {\em Large-$M$ limit of $p_1(M)$.} The (pointwise) limit of the functions $p\mapsto-\osV_0^c(F(p))$, as $M\to \infty$, is $$
      v^c(p):=-(1-p)(p^2+p-1)=p^3-2p+1, \quad p\in(0,1],
      $$
      which is the red curve depicted in Figure~\ref{fig:V1}~(right). It turns out to coincide (up to a factor $p$) with the {\em mean Euler characteristic per site} $\oV_0(\Z^{2,\text{NN}},p)$ of site percolation on the nearest neighbor graph on $\Z^2$ as a function of $p\in[0,1]$, see e.g.~\cite[eq.\ (5), p.\ 4]{NMW08}. In particular, one gets for the zeros that
      \begin{align*}
        \lim_{M\to\infty} p_1(M)=p_{0,\text{NN}},
      \end{align*}
      where $p_{0,\text{NN}}=(\sqrt{5}-1)/2\approx 0.618$ is the unique zero of $v^c$ in $(0,1)$. %$p\mapsto\oV_0^{site,4}(p)$.
      Note that $p_{0,\text{NN}}$ is strictly larger than $p_{c,\text{NN}}\approx 0.59$. Thus for large $M$, $p_c(M)<p_1(M)$, while for $M=2,3$, one has $p_c(M)>p_1(M)$. So $p_1$ can neither be a general lower bound nor a general upper bound for $p_c$. This observation also rules out the minimum of $p\mapsto \osV^c_0(F(p))$ to be a good general bound in any way.

      These findings show that there is not such a close connection between the Euler characteristics and percolation thresholds in this fractal model as there are in other percolation models. An explanation, why the phase transition leaves no signature in the studied functionals might be that percolation happens in fact on lower dimensional subsets. Recently it has been shown, see \cite{BCJM13}, that for $p\geq p_c$ (and conditioned on $F$ being nonempty), the union $Z$ of all connected components of $F$ larger than one point forms almost surely a set of strictly smaller Hausdorff dimension than the remaining set $F\setminus Z$ (the dust), which has dimension $\dim_H F\setminus Z=\dim_H F=D$ almost surely. The rescaling with $r^{Dn}$ of the geometric functionals essentially means that they do not see the lower dimensional set $Z$ on which percolation occurs. So from the point of view of the Hausdorff dimension, our result is consistent with the findings in \cite{BCJM13}. But in \cite{BCJM13}, it is also shown that in contrast the Minkowski dimensions of $Z$ and $F\setminus Z$ coincide almost surely for $p\geq p_c$. Since our approximation of $F$ by unions of boxes $F_n$ is %in fact a box counting approach, it is
      rather related to the Minkowski (or box) dimension than to the Hausdorff dimension, our results support the hypothesis that, also in the Minkowski setting, the effect of the dust dominates that of the larger components, though not on the level of dimension but on the refined level of associated measures or contents as provided by our functionals. Long before percolation occurs (i.e.\ for $p<p_c$), the expected Euler characteristic $\E V_0(F_n)$ becomes negative, i.e.\ it detects more holes than components in the approximations $F_n$, which indicates that the $n$-th approximation of the dust must have a lot of structure which only disappears in the limit. More refined methods are necessary to separate the dust from the larger clusters. It might for instance be worth to look at the Euler characteristic of the percolation cluster in finite approximations.

\section{{Outlook and outline}}
      We emphasize that, although our work is motivated by questions regarding the percolation properties,
      our focus here is on establishing the existence of the geometric limit functionals $\osV_k(F)$ and $\osV_k^c(F)$ (Theorems~\ref{thm:Vk-limit-general} and \ref{thm:Zk-limit-general2}), and on computing them explicitly in dimension 1 (Corollaries~\ref{cor:Vk-dim1} and \ref{cor:5.2}) and 2 (Theorems~\ref{thm:Vk-limit-dim2}, \ref{thm:Vck-limit-dim2} and Proposition~\ref{prop:Vck-limit-dim2}).
      The methods developed %and the results obtained
      here can be transferred to other random (self-similar) models. The functionals may have other applications. Just as fractal curvatures, they clearly carry geometric information beyond the fractal dimension, but unlike them, they can be computed explicitly for random sets (at least in some cases). Even more importantly, they can be estimated well from the finite approximations, see Remarks~\ref{rem:speed} and \ref{rem:speed2} for a discussion of the speed of convergence of $r^{Dn}\E V_0(F_n)$ and $r^{Dn}\E V_0(C_n)$ as $n\to\infty$, and see Remark~\ref{rem:simu} for a practical demonstration. Hence the functionals may serve as robust and efficient geometric descriptors in applications and may e.g.\ help to distinguish different geometric structures of the same fractal dimension. It is an aim of future research to develop this ``box counting'' approach {to geometric functionals} further to work for general (random) fractals.
      {In view of the fast convergence of these functionals it is another
intriguing question whether a similar speed of convergence can be expected for the percolation probabilities of $F_n$.} %, whose construction is not necessarily based on boxes. In future research, this ``box counting'' approach can be developed both theoretically and in applications to describe general random fractals.
       {The functionals $\osV_k(F)$ describe first order properties of the random set $F$ in a sense that can be made precise: for $p>M^{-d}$, the almost sure convergence, as $n\to\infty$, of the random variables  $r^{n(D-k)} V_k(F_{n})$ to some limit random variable $Z^k_\infty$ can be shown, and the limit of expectations $\osV_k(F)$ appears as the expectation of this limit variable $Z^k_\infty$. Moreover, the functionals $\osV_k(F)$ turn out to determine essentially the covariance structure and also higher moments of the limit variables $Z^k_\infty$. We discuss these results further in \cite{KW20}.}
      \paragraph{Outline.} The remainder of this article is organized as follows. In Section~\ref{sec:2}, we describe fractal percolation as a random self-similar set and introduce some notation and basic concepts. In Section~\ref{sec:3}, we study in detail the approximation of $F$ by the sets $F_n$, and in Section~\ref{sec:5} the approximation by the sets $C_n$. In both cases we prove first a general formula for arbitrary dimensions (Theorems~\ref{thm:Vk-limit-general} and \ref{thm:Zk-limit-general2}) which we then use to compute the limit functionals in $\R$ and $\R^2$. A careful analysis of the model in $\R$ is essential for the computations in $\R^2$. Additionally, it is necessary to understand the intersection of two independent copies of $F$ in $\R$, the analysis of which also provides a new point of view on the lower bound for $p_c$ in \eqref{eq:pc-bounds} obtained in \cite{CCD88}, see Remark~\ref{rem:K1capK2}. In the course of the proofs not only explicit expressions for the limit functionals are derived but also exact formulas for the $n$-th approximations, see in particular Remarks~\ref{rem:speed} and \ref{rem:speed2}. {Following another idea in the physics literature \cite{Knuefing05,MeckeSchoenhoefer15}, these formulas are also used to derive in Remark~\ref{rem:subdim} the \emph{fractal subdimensions} and the exact associated \emph{amplitudes} for this model, which is an alternative approach to refined information about fractal sets beyond fractal dimension.}

      Finally, in Section~\ref{sec:appendix} some estimates are proved which ensure the convergence of the series occurring in the main formulas in Theorems~\ref{thm:Vk-limit-general} and \ref{thm:Zk-limit-general2}. They are not needed for the further results in $\R$ and $\R^2$, as the convergence can be checked directly in these cases but ensure their validity in higher dimensions.

\section{Fractal percolation as a random self-similar set} \label{sec:2}
Fractal percolation $F$ in $\R^d$ with parameters $p\in[0,1]$ and $M\in\N_{\geq 2}$ is a random self-similar set generated by the following random iterated function system (RIFS) $\sS$ constructed on the basic set $J=[0,1]^d$. {Denote by $J_1,\ldots,J_{M^d}$ the $M^d$ subcubes of sidelength $r=1/M$ into which $J$} is divided in the first step of the construction of $F$ described above. $\sS$ is a random subset of the set $\Phi:=\{\phi_1,\ldots, \phi_{M^d}\}$, where $\phi_j$, $j=1,\ldots, M^d$, is the similarity which maps $J$ to $J_j$ (rotation and reflection free, for simplicity and uniqueness).
Each map $\phi_j$ is included in $\sS$ with probability $p$ independent of all the other maps. It is obvious that $\sS$ satisfies the open set condition (OSC) with respect to the interior $\intr(J)$ of $J$, since $\sS$ is a random subset of $\Phi$ and even the full set $\Phi$ satisfies OSC with respect to $\intr(J)$.
%, cf.\ the conditions in \eqref{eq:OSC}. Note that each $\phi_j$ has the same contraction ratio $r=1/M$. %It is convenient to assume that the $\phi_j$ do not include any rotations. Then each $\phi_j$ is of the form $\phi_j(x)=r(x+b_j)$, $x\in\R^d$ for some translation vector $b_j\in\{0,\ldots,M-1\}^d$.

For obtaining $F$ as an invariant set of the RIFS $\sS$, we employ a Galton-Watson tree on the set of all finite words $\Sigma_*:=\bigcup_{n=0}^\infty \Sigma_n$, where $\Sigma_n:=\{1,\ldots,M^d\}^n$, $n\in\N_0$. In particular, $\Sigma_0=\{\eps\}$ %over the alphabet $\{1,\ldots,M^d\}$ (including the empty word
where $\eps$ is the empty word of length $|\eps|=0$. For each $\sigma\in\Sigma_*$, let $\sS_\sigma$ be an independent copy of the RIFS $\sS$.
$\sS_\sigma$ contains a random number $\nu_\sigma$ of maps (with $\nu_\sigma$ being binomially distributed with $p$ and $M^d$%taking values in $\{0,1,\ldots,M^d\}$
). Let $I_\sigma\subseteq \{1,\ldots,M^d\}$ be the set of indices of the maps in $\sS_\sigma$. It is convenient to denote these maps by $\phi_{\sigma i}$, $i\in I_\sigma$. Note that $|I_\sigma|=\nu_\sigma$. In particular, $I_\sigma$ may be empty.
We build a random tree $\sT$ in $\Sigma_*$ as follows:
set $\sT_0:=\{\eps\}$ and define, for $n\in\N_0$, $\sT_{n+1}:=\emptyset$, if $\sT_n=\emptyset$, and
$$
\sT_{n+1}:=\{\sigma i: \sigma\in\sT_n, i\in I_\sigma\},
$$
if $\sT_n\neq\emptyset$. Finally, we set
$$
\sT:=\bigcup_{n=0}^\infty \sT_n.
$$
$\sT$ can be interpreted as the population tree of a Galton-Watson process in which $\sT_n$ represents the $n$-th generation and $\sigma i\in\sT_{n+1}$, $i\in I_\sigma$ are the decendants of an individuum $\sigma\in\sT_n$.
The \emph{self-similar random set} associated with the RIFS $\sS$ is the set
$$
F:=\bigcap_{n=1}^\infty \bigcup_{\sigma\in\sT_n} J_\sigma,
$$
where, for any $\sigma\in\Sigma_*$ of length $|\sigma|=n\in\N$ and any set $K\subset\R^d$,
$$
K_\sigma:=%\bar \phi_{\sigma}(K):=
\phi_{\sigma|1}\circ \phi_{\sigma|2}\circ \ldots \circ \phi_{\sigma|n}(K).
$$
Here $\sigma|k$, $k\in\{1,\ldots,n\}$ denotes the word formed by the first $k$ letters of $\sigma$.
$F$ is called self-similar because of the following stochastic self-similarity property (which characterizes $F$ uniquely): if $F^{(i)}$, $i\in\{0,1,\ldots,M^d\}$ are i.i.d.\ copies of $F$ and $\sS$ is the corresponding RIFS as above, independent of the $F^{(i)}$, then
$$
F^{(0)}=\bigcup_{\phi_i\in\sS} \phi_i(F^{(i)}).
$$

In the language of the tree and the associated sets considered above, the construction steps $F_{n}$, $n\in\N$ of the fractal percolation process are given by
$$
F_{n}=\bigcup_{\sigma\in\sT_n} J_\sigma.
$$
Here the sets $J_\sigma$, with $|\sigma|=n$ encode the subcubes of level $n$ of the construction, and the above union extends over those subcubes $J_\sigma$, which survived all the previous steps, i.e.\ over all $\sigma$ for which all the cubes $J_{\sigma|i}$, $i\in\{1,\ldots, n\}$, have been kept in the $i$-th step of the construction.
We also introduce, for each $j\in\{1,\ldots,M^d\}$ and each $n\in\N$, the set
\begin{align} \label{eq:Fj_ndef}
F_n^j:=\bigcup_{{\sigma\in\sT_n},{\sigma|1=j}} J_\sigma,
\end{align}
being the union of those cubes of level $n$ which are subcubes of $J_j=\phi_j(J)$.
We will not make much use of the limit objects and their self-similarity in the sequel, we will mainly use the following basic properties of the construction steps $F_n$ and their parts $F_n^j$:
For any $j\in\{1,\ldots,M^d\}$ and any $n\in\N$, we have
\begin{align}
  \label{eq:basic-sim} F_n^j=\phi_j(\widetilde F_{n-1})
\end{align}
in distribution, where $\widetilde F_{n-1}$ is the random set which equals $F_{n-1}$ with probability $p$ and is empty otherwise (i.e.,\ $\widetilde F_{n-1}^j= F_{n-1}^j\cap \widetilde J_j$, where $\widetilde J_j$ is a random set independent of $F_{n-1}^j$, which equals $J_j$ with probability $p$ and is empty otherwise).
The homogeneity and motion invariance of the intrinsic volumes implies now in particular that
 \begin{align}
   \label{eq:basic-sim-2}
   \E V_k(F_n^j)=p \E V_k(\phi_j(F_{n-1}))=p r^k \E V_k(F_{n-1}),
 \end{align}
for any $k\in\{0,\ldots,d\}$ and any $j\in\{1,\ldots,M^d\}$, where $r=1/M$ is the scaling ratio of $\phi_j$.

\section{Approximation of $F$ by the sequence $(F_{n})_n$.} \label{sec:3}

Our first aim in this section is to prove Theorem~\ref{thm:Vk-limit-general}. {Let $M\in\N_{\geq 2}$ and $p\in(0,1]$ be arbitrary and
let $D$ be as defined in \eqref{eq:dimF}.} ($D$ is the Minkowski dimension of $F$ in case $M^dp\geq 1$ and negative otherwise.)
Set
\begin{align}\label{eq:vkdef}
\overline{v}_k(n):=r^{n(D-k)} \E V_k(F_{n}), \quad n\in\N_0,
\end{align}
where $F_{0}:=J=[0,1]^d$. Since the latter is a deterministic set, we have $\overline{v}_k(0)= V_k(F_{0})=q_{d,k}$. We are going to show that the limit $\osV_k(F)=\lim_{n\to\infty} \overline{v}_k(n)$ exists for any $k$ and, moreover, that it coincides with the expression stated in \eqref{eq:Vk-limit-general}. The first step is to derive a kind of renewal equation for the $\overline{v}_k$. (The approach is similar to the methods in \cite{W08,Zaehle11} which are based based on renewal theory. However, here we do not need the Renewal theorem as it is possible to argue directly.)

Setting
$$
w_k(n):= \overline{v}_k(n)-\overline{v}_k(n-1), \quad n\in\N,
$$
it is easy to see that
\begin{align} \label{eq:Z-w-relation}
\lim_{n\to\infty} \overline{v}_k(n)=\overline{v}_k(0)+\sum_{j=1}^\infty w_k(j),
\end{align}
i.e., the limit on the left exists if and only if the sum on the right converges.
(Indeed, by definition of $w_k$, we have $\overline{v}_k(n)=\overline{v}_k(n-1)+w_k(n)=\ldots =\overline{v}_k(0)+\sum_{j=1}^n w_k(j)$ for any $n\in\N$, and so in \eqref{eq:Z-w-relation} the limit on the left exists if and only if the partial sums on the right converge.) Therefore, it is enough to compute the functions $w_k$, which turns out to be easier than computing the $\overline{v}_k$ directly.
The relation
$$
\overline{v}_k(n)=\overline{v}_k(n-1)+w_k(n),\quad  n\in\N
$$
can be viewed as a (discrete) renewal equation with $w_k$ being the error term.
By definition of $w_k$, we have
\begin{align*} %\label{eq:wk}
 \notag w_k(n)&= \overline{v}_k(n)-\overline{v}_k(n-1)= r^{n(D-k)} \E V_k(F_{n})- r^{(n-1)(D-k)} \E V_k(F_{n-1})\\
 \notag &= r^{n(D-k)}\left( \E V_k(F_{n})- r^{k-D} \E V_k(F_{n-1})\right)\\
  &= r^{n(D-k)}\left( \E V_k(F_{n})- M^dp \, r^{k} \E V_k(F_{n-1})\right),
 % &= r^{n(D-k)}\left( \E V_k(F_{n})- M^dp \E V_k(F_n^1)\right),\\
%  &= r^{n(D-k)}\left( \E V_k(F_{n})- \sum_{j} \E V_k(F_n^j)\right),\\
\end{align*}
where we employed the relation $M^dp=r^{-D}$ in the last step.
Now the similarity relation \eqref{eq:basic-sim-2} implies
$
\sum_{j=1}^{M^d} \E V_k(F_n^j)=M^d p r^k \E V_k(F_{n-1}),
$
which we can insert in the above expression to obtain
\begin{align} \label{eq:wk2}
  w_k(n)&= r^{n(D-k)}\left( \E V_k(F_{n})- \sum_{j=1}^{M^d} \E V_k(F_n^j)\right).
 % &= r^{n(D-k)}\left( \E V_k(F_{n})- M^dp \E V_k(F_n^1)\right),\\
%  &= r^{n(D-k)}\left( \E V_k(F_{n})- \sum_{j} \E V_k(F_n^j)\right),\\
\end{align}
Using the inclusion-exclusion principle, this can be expressed in a more convenient form.
 %Given any $n\in \N$ and any realization $F_{n}(\omega)$ of the random set $F_{n}$, we set $F^j_{n}(\omega):=\bigcup_{\sigma\in\sT_n(\omega), \sigma|1=j}J_\sigma$, $j=1,\ldots, N$. %(Recall that the $\phi_j$ are maps in the underlying IFS.)
%Then
Since
$
F_{n}=\bigcup_{j=1}^{N} F^j_{n},
$
%and therefore, by the inclusion-exclusion principle,
we get
$$
V_k(F_{n} )-\sum_{j=1}^{N} V_k(F^j_{n} )=\sum_{T\subset\{1,\ldots,M^d\},|T|\geq 2}(-1)^{|T|-1} V_k(\bigcap_{j\in T} F^j_{n} ).
$$
Taking expectations and plugging the resulting equation into \eqref{eq:wk2}, we obtain
for each $n\in\N$ and each $k\in\{0,\ldots,d\}$ the representation
\begin{align}\label{eq:wk-rep}
  w_k(n)&=  \sum_{T\subset\{1,\ldots,M^d\},|T|\geq 2}(-1)^{|T|-1} r^{n(D-k)} \E V_k(\bigcap_{j\in T} F_n^j).
 % &= r^{n(D-k)}\left( \E V_k(F_{n})- M^dp \E V_k(F_n^1)\right),\\
%  &= r^{n(D-k)}\left( \E V_k(F_{n})- \sum_{j} \E V_k(F_n^j)\right),\\
\end{align}
Note that this is a finite sum with a fixed number of terms (independent of $n$). Combined with \eqref{eq:Z-w-relation}, it yields
\begin{align}\label{eq:proof-1-1}
 \osV_k(F) &=  q_{d,k}+\sum_{n=1}^\infty \sum_{T\subset\{1,\ldots,M^d\},|T|\geq 2}(-1)^{|T|-1} r^{n(D-k)} \E V_k(\bigcap_{j\in T} F_n^j).
\end{align}
This is almost the formula stated in Theorem~\ref{thm:Vk-limit-general} except for the different order of summation. The summations can be interchanged (and thus the formula \eqref{eq:Vk-limit-general} is verified) provided that the summations over $n$ in \eqref{eq:Vk-limit-general} converge for each set $T$. This convergence is ensured by Proposition~\ref{prop:Rk-conv} below
{for any $p\in(0,1]$}.
Recall that the $k$-th intrinsic volume of a polyconvex set $K$ can be localized to a signed measure on $K$, the $k$-th curvature measure $C_k(K,\cdot)$. Denote by $C_k^{\var}(K)$ the total mass of the total variation measure of $C_k(K,\cdot)$.%, completing the proof of Theorem~~\ref{thm:Vk-limit-general}.

\begin{prop}\label{prop:Rk-conv}
  Let $F$ be a fractal percolation in $[0,1]^d$ with parameters $M\geq 2$ and {$p\in(0,1]$}. %(M^{-\min\{3,d\}},1]$.
  For each $k\in\{0,\ldots,d\}$ and each $T\subset\{1,\ldots, M^d\}$ with $|T|\geq 2$,
  $$
  \sum_{n=1}^\infty r^{n(D-k)} \E C_k^\var(\bigcap_{j\in T} F_n^j)<\infty.
  $$
  In particular, the sums
  $$
  \sum_{n=1}^\infty r^{n(D-k)} \E V_k(\bigcap_{j\in T} F_n^j)
  $$
  converge absolutely.
\end{prop}
We postpone the proof of Proposition~\ref{prop:Rk-conv} to the last section where we will discuss it together with the proof of a similar assertion needed in Section~\ref{sec:5}.
With this statement at hand we can now complete the proof of the main theorem.
\begin{proof}[Proof of Theorem~\ref{thm:Vk-limit-general}]
  %Plugging the representation \eqref{eq:wk-rep} for $w_k$ into equation \eqref{eq:Z-w-relation}, yields
%  \begin{align*}
%    \lim_{n\to\infty} \overline{v}_k(n)&=\overline{v}_k(0)+\sum_{j=1}^\infty \sum_{T\subset\{1,\ldots,M^d\},|T|\geq 2}(-1)^{|T|-1} r^{n(D-k)} \E V_k(\bigcap_{j\in T} F_n^j)\\
%    &=\overline{v}_k(0)+\sum_{T\subset\{1,\ldots,M^d\},|T|\geq 2}(-1)^{|T|-1} \sum_{j=1}^\infty  r^{n(D-k)} \E V_k(\bigcap_{j\in T} F_n^j),
%  \end{align*}
To obtain the formula \eqref{eq:Vk-limit-general}, all we have to do is to interchange the order of the summations in  the formula \eqref{eq:proof-1-1}. This is justified, since, by Proposition~\ref{prop:Rk-conv}, %for $p>M^{-\min\{3,d\}}$
{all the series occurring in \eqref{eq:Vk-limit-general} converge for any $p\in(0,1]$.}
%, showing the existence of the limit functionals $\osV_k(F)$.
\end{proof}

{Now, we are going to apply formula \eqref{eq:Vk-limit-general} to
derive explicit expressions for the functionals $\osV_k(F)$ for fractal
percolation $F$ in $\R^d$ for dimensions $d=1,2$.
In particular, we will prove Theorem~\ref{thm:Vk-limit-dim2}.
The computations in dimension $d=2$ require explicit formulas for the
$n$-th construction steps of the one-dimensional case.
The same method can, in principle, provide explicit expressions in
dimension $d=3$ and higher.
The derivation in dimension $d$ relies on formulas for all dimensions up
to $d-1$ for the $n$-th construction steps $F_n$ of $F$ (and for
intersections of their independent copies).
While the computations quickly become technically involved,
including separate analyses of many cases of intersections
and lengthy expressions, the method remains the same.
In this sense the case $d=2$ is prototypic.}

\paragraph*{The case $d=1$.} Fractal percolation in one dimension {(for which we use throughout the letter $K$ instead of $F$)} is not very interesting as a percolation model. However, the limiting behaviour of the studied geometric functionals is of independent interest. Moreover{, as indicated above,} the case $d=1$ is essential for the computations in the two-dimensional case. First, we derive explicit expressions for the expected intrinsic volumes of the approximation steps $K_{n}$, $n\in\N_0$ of a fractal percolation $K$ in $[0,1]$, from which it is easy to determine the rescaled limits $\osV_k(K)$. Then we study the intersection of two such random sets{, cf.\ Proposition~\ref{prop:MPcapMP}, which is needed too for the discussion of the case $d=2$}.

\begin{prop} \label{prop:MP}
Let $K$ be a fractal percolation on the interval $[0,1]$ with parameters $M\in\N_{\geq 2}$ and {$p\in(0,1]$}. Denote by $K_{n}$ the $n$-th step of the construction of $K$.
Then, for any $n\in\N_0$,
  \begin{align*}
    \E &V_1(K_{n})=p^{n} \qquad \text{ and }\\
    \E &V_0(K_{n})=%(Mp^2)^n\left(3-2 M^{-n}-4p \frac{M-1}{M-p}\left(1-\left(\frac pM\right)^n\right)\right).
    (Mp)^n\left(1-\frac{(M-1)p}{M-p}\left[1-\left(\frac {p}M\right)^n\right]\right).
  \end{align*}
\end{prop}
  \begin{proof}
    For $j=1,\ldots,M$, let $K_{n}^j$ be the union of the surviving intervals of level $n$ contained in $J_j=\phi_j([0,1])$, cf.~\eqref{eq:Fj_ndef}. Then $K_{n}=\bigcup_{j=1}^M K_{n}^j$ and since in this union only sets $K_{n}^j$ with consecutive indices can have a nonempty intersection, by the inclusion-exclusion formula, we get
    \begin{align} \label{eq:propMP1-1}
      \E V_k(K_{n})=\sum_{j=1}^M \E V_k(K_{n}^j)-\sum_{j=1}^{M-1} \E V_k(K_{n}^j\cap K_{n}^{j+1}).
    \end{align}
    For $k=1$, the second sum vanishes, since these intersections consist of at most one point. Moreover, by \eqref{eq:basic-sim-2}, the terms in the first sum satisfy
    \begin{align} \label{eq:propMP1-2}
      \E V_k(K_{n}^j)=pr^k\E V_k(K_{n-1}),\quad n\in\N.
    \end{align}
    Since $V_1(K_0)=V_1([0,1])=1$, this yields
    \begin{align*}
    \E V_1(K_{n})=\sum_{j=1}^M (p/M) \E V_1(K_{n-1})=p \E V_1(K_{n-1})=\ldots = p^n
    \end{align*}
    as claimed. For $k=0$, the terms in second sum in \eqref{eq:propMP1-1} contribute. The Euler characteristic $V_0(K_{n}^j\cap K_{n}^{j+1})$ equals $1$ with probability $p^{2n}$ (and is $0$ otherwise), since for a nonempty intersection at each level from 1 to $n$ the two intervals containing the possible intersection point need to survive (which has probability $p$ for each of these intervals).
    Using this and \eqref{eq:propMP1-2}, we conclude from \eqref{eq:propMP1-1} that
    \begin{align*} %\label{eq:propMP1-3}
      \E V_0(K_{n})=\sum_{j=1}^M p \E V_0(K_{n-1})-\sum_{j=1}^{M-1} p^{2n}=Mp \E V_0(K_{n-1})-(M-1)p^{2n}.
    \end{align*}
    This is a recursive relation for the sequence $(\E V_0(K_{n}))_{n\in\N_0}$ where $\E V_0(K_0)=1$. By an induction argument, it is easy to obtain the explicit representation
    \begin{align*}
       \E V_0(K_{n})=(Mp)^n-(M-1)\sum_{i=1}^n (Mp)^{n-i} p^{2i}
       =(Mp)^n\left(1-(M-1)\sum_{i=1}^n \left(\frac pM\right)^i\right),
       %=(Mp)^n\left(1-p\frac{M-1}{M-p}\left[1-\left(\frac pM\right)^n\right]\right),
    \end{align*}
    which yields the asserted formula.
  \end{proof}

  \begin{cor} \label{cor:Vk-dim1}
    Let $K$ be a fractal percolation on the interval $[0,1]$ with parameters $M\in\N_{\geq 2}$ and {$p\in(0,1]$}. Then
    \begin{align*}
      \osV_1(K)&=1 \qquad \text{ and }\qquad \osV_0(K)=\frac{M(1-p)}{M-p}.
    \end{align*}
  \end{cor}
  \begin{proof}
    Since $D=\frac{\log (Mp)}{\log M}$, we have $Mp=M^D=r^{-D}$ and so, by Proposition~\ref{prop:MP},
    \begin{align*}
      \osV_0(K)=\lim_{n\to\infty} r^{Dn} \E V_0(K_{n})=\lim_{n\to\infty} 1-\frac{(M-1)p}{M-p}\left[1-\left(\frac {p}M\right)^n\right]=1-\frac{(M-1)p}{M-p}
    \end{align*}
    and $\osV_1(K)=\lim_{n\to\infty} r^{(D-1)n} \E V_1(K_{n})=\lim_{n\to\infty} p^{-n} p^n=1$, as claimed.
  \end{proof}

  Figure~\ref{fig:KcapK} (left) shows plots of $\osV_0(K)$ as a function of $p$ for different parameters $M$. It is apparent that these are positive and monotone decreasing functions in $p$ for any $M$ and that the limit as $M\to\infty$ is given by
$f(p)=1-p$.

\begin{prop} \label{prop:MPcapMP}
Let $K^{(1)}, K^{(2)}$ be independent fractal percolations on the interval $[0,1]$ with the same parameters $M\in\N_{\geq 2}$ and {$p\in(0,1]$}.
Then, for any $n\in\N_0$,
  \begin{align*}
    \E &V_1(K_n^{(1)}\cap K_n^{(2)})=p^{2n} \qquad \text{ and }\\
    \E &V_0(K_n^{(1)}\cap K_n^{(2)})=%(Mp^2)^n\left(3-2 M^{-n}-4p \frac{M-1}{M-p}\left(1-\left(\frac pM\right)^n\right)\right).
    (Mp^2)^n\times\\&\left(3-2 M^{-n}
-4p \frac{M-1}{M-p}\left[1-\left(\frac pM\right)^n\right]
+\frac{(M-1)p^2}{M-p^2}\left[1-\left(\frac {p^2}M\right)^n\right]\right).
  \end{align*}
\end{prop}
\begin{proof}
    %$K^{(1)}$ and $K^{(2)}$ are self-similar random sets in $\R$ on the basic set $J=[0,1]$ with a generating random tree $\sT^i$, $i=1,2$, in the background (as described in Section~\ref{sec:2}). Let $\sT^i_n$, $n\in\N$ denote the $n$-th generation in $\sT^i$.
%   We can partition each $K_n^{(i)}$ into $M$ parts corresponding to the $M$ subintervals $J_j=\phi_j(J)$, $j=\{1,\ldots,M\}$ of the first subdivision in the construction:
%  For each $j\in\{1,\ldots,M\}$ (and each $n\in\N$), let
For $i\in\{1,2\}$ and $j\in\{1,\ldots,M\}$, let
$K_n^{(i),j}%:=\bigcup_{{\sigma\in\sT^i_n},{\sigma|1=j}} J_\sigma, \quad i=1,2
$
be the union of those level-$n$ intervals in the union $K_{n}^{(i)}$ which are contained in $J_j$ (similarly as in \eqref{eq:Fj_ndef}).
Then $K_n^{(i)}=\bigcup_{j=1}^M K_n^{(i),j}$. Since $K_i^j\subset J_j$ and $J_j\cap J_l\neq\emptyset$ if and only if $|j-l|\le 1$, we can write the intersection $K_n^{(1)}\cap K_n^{(2)}$ as
$$
K_n^{(1)}\cap K_n^{(2)}=\bigcup_{j=1}^M K_n^{(1),j} \cap \bigcup_{l=1}^M K_n^{(1),l}=\bigcup_{j=1}^M \left( K_n^{(1),j}\cap \bigcup_{l=j-1}^{j+1}  K_n^{(2),l}\right)=:\bigcup_{j=1}^M L_j,
$$
where we have set $K_n^{(2),0}=K_n^{(2),M+1}:=\emptyset$ for convenience. The random sets $L_j$ (whose dependence on $n$ we suppress in the notation) satisfy $L_j\subset J_j$ a.s.\ and therefore in the union $\bigcup_j L_j$ only sets with consecutive indices can have a nonempty intersection. Thus, by the inclusion-exclusion principle, we conclude for the expected intrinsic volumes
\begin{align} \label{eq:K1capK2}
  \E V_k(K_n^{(1)}\cap K_n^{(2)})=\sum_{j=1}^M \E V_k(L_j)- \sum_{j=1}^{M-1} \E V_k(L_j\cap L_{j+1}).
\end{align}
Now observe that, for $j=1,\ldots,M-1$,
$$
L_j\cap L_{j+1}= K_n^{(1),j}\cap K_n^{(1),j+1}\cap \left(K_n^{(2),j}\cup K_n^{(2),j+1}\right),
$$
and this random set is either empty or consists of exactly one point $z_j$ (namely, the unique point in the intersection $J_j\cap J_{j+1}$). The latter event occurs if and only if for each of the two sets $K_n^{(1),j}, K_n^{(1),j+1}$ at each level $k=1,\ldots,n$ the subinterval of level $k$ that contains $z_j$ survives (which has each probability $p^n$) and if a similar survival of all subintervals containing $z_j$ also occurs for at least one of the sets  $K_n^{(2),j}, K_n^{(2),j+1}$. The probability for this latter event is $2p^{n}-p^{2n}$. Hence $\E V_k(L_j\cap L_{j+1})=\left(2p^{3n}-p^{4n}\right) V_k(\{z_j\})$ and therefore
\begin{align} \label{eq:LcapL}
\E V_0(L_j\cap L_{j+1})=2p^{3n}-p^{4n} \quad  \text{ and } \quad \E V_k(L_j\cap L_{j+1})=0, \text{ for } k\ge 1.
\end{align}
It remains to determine $\E V_k(L_j)$. By definition of $L_j$, we have
\begin{align*}
  L_j= \bigcup_{l=j-1}^{j+1} K_n^{(1),j}\cap K_n^{(2),l}
\end{align*}
and therefore the inclusion-exclusion formula gives
 \begin{align*}
 \E V_k( L_j)= \sum_{l=j-1}^{j+1} \E V_k (K_n^{(1),j}\cap K_n^{(2),l})-\sum_{l=j-1}^j \E V_k(K_n^{(1),j}\cap K_n^{(2),l}\cap K_n^{(2),l+1}).
\end{align*}
Now again $K_n^{(1),j}\cap K_n^{(2),l}$ is a singleton with probability $p^{2n}$ and empty otherwise, provided $l=j-1$ or $l=j+1$ (and $l\notin\{0,M+1\}$). Similarly,  $K_n^{(1),j}\cap K_n^{(2),l}\cap K_n^{(2),l+1}$ is a singleton with probability $p^{3n}$ and empty otherwise, provided $l\notin\{0,M\}$. (For the exceptional $l$, these intersections are empty a.s.) This implies
\begin{align*}
 \E V_0( L_j)= \E V_0(K_n^{(1),j}\cap K_n^{(2),j})+ \left\{\begin{array}{cc}
    2(p^{2n}-p^{3n}), &  j\in\{2,\ldots,M-1\},\\
    p^{2n}-p^{3n}, &  j\in\{1,M\},
 \end{array} \right.
\end{align*}
and $\E V_k(L_j)= \E V_k(K_n^{(1),j}\cap K_n^{(2),j})$ for any $k\geq 1$.
Plugging this and \eqref{eq:LcapL} into equation \eqref{eq:K1capK2}, we conclude that, for $k=0,1$ and any $n\in\N$,
\begin{align} \label{eq:K1capK2-2}
  \E V_k(K_n^{(1)}\cap K_n^{(2)})=p_k(n)+\sum_{j=1}^M \E V_k(K_n^{(1),j}\cap K_n^{(2),j}),
\end{align}
where $p_0(n):=(M-1)(2p^{2n}-4p^{3n}+p^{4n})$ %=2(M-1)p^{2n}(1-2p^n)$
and $p_1(n):=0$, $n\in\N.$
Now observe that, by \eqref{eq:basic-sim}, we have $K_n^{(1),j}\cap K_n^{(2),j}=\phi_j(\widetilde K_{n-1}^{(1)}\cap \widetilde K_{n-1}^{(2)})$ in distribution, where $\widetilde K_{n-1}^{(i)}$ is similarly as in \eqref{eq:basic-sim} the random set which equals $K_{n-1}^{(i)}$ with probability $p$ and is empty otherwise. This implies
\begin{align*}
   \E V_k(K_n^{(1),j}\cap K_n^{(2),j})= p^2 r^k \E V_k(K_{n-1}^{(1)}\cap K_{n-1}^{(2)}),
\end{align*}
for any $j=1,\ldots,M$ and any $n\in\N$, where $K_{0}^{(i)}=[0,1]$ and thus $\E V_k(K_{0}^{(1)}\cap K_{0}^{(2)})=V_k([0,1])=1$ for $k=0,1$.
Setting $\alpha_n:=\E V_1(K_n^{(1)}\cap K_n^{(2)})$, $n\in\N_0$, we have $\alpha_0 =1$ and we infer from \eqref{eq:K1capK2-2} that
$$
\alpha_n=\sum_{j=1}^M \E V_1(K_n^{(1),j}\cap K_n^{(2),j})=M p^2 r \alpha_{n-1}=p^2 \alpha_{n-1}, \quad n\in\N.
$$
It is easy to see now that $\alpha_n=p^{2n}$, proving the first formula in Proposition~\ref{prop:MPcapMP}.

Setting $\beta_n:=\E V_0(K_n^{(1)}\cap K_n^{(2)})$, $n\in\N_0$, we infer in a similar way that $\beta_0=1$ and
$$
\beta_n=\sum_{j=1}^M \E V_0(K_n^{(1),j}\cap K_n^{(2),j})+ p_0(n)=M p^2 \beta_{n-1}+ p_0(n), \quad n\in\N,
$$
which provides a recursive relation for the sequence $(\beta_n)_n$. By an induction argument, we obtain
$$
\beta_n=(Mp^2)^n+ \sum_{j=1}^n (Mp^2)^{n-j} p_0(j), \quad n\in\N_0.
$$
Plugging in the $p_0(j)$ and computing the sum, we conclude that, for any $n\in\N_0$,
$$
\beta_n=(Mp^2)^n\left(3-\frac 2{M^n}
-4p \frac{M-1}{M-p}\left[1-\left(\frac pM\right)^n\right]
+\frac{(M-1)p^2}{M-p^2}\left[1-\left(\frac {p^2}M\right)^n\right]\right),
$$
which shows the second formula in Proposition~\ref{prop:MPcapMP} and completes the proof.
\end{proof}

\begin{rem} \label{rem:1dim-Klim}
   %Letting $D:=\log(Mp^2)/\log(M)$, which is the Minkowski dimension of $K^{(1)}\cap K^{(2)}$,
   It is easy to see from  Proposition~\ref{prop:MPcapMP} that for $D':=\log(Mp^2)/\log M$ the rescaled expressions $r^{n(D'-k)}\E V_k(K_n^{(1)}\cap K_n^{(2)})$ converge as $n\to\infty$. Indeed, since $r^{D'-1}=p^{-2}$ and $r^{D'}=(1/M)^{D'}=(M p^2)^{-1}$, we obtain
   \begin{align*}
   \osV_1(K^{(1)}\cap K^{(2)})&:=\lim_{n\to\infty} r^{n(D'-1)}\E V_1(K_n^{(1)}\cap K_n^{(2)})=1
   \qquad \text{ and }\\
   \osV_0(K^{(1)}\cap K^{(2)})&:=\lim_{n\to\infty} r^{nD'}\E V_0(K_n^{(1)}\cap K_n^{(2)})%=\lim_{n\to\infty} (Mp^2)^{-n}\beta_n\\
   %&=\lim_{n\to\infty} 3-2 M^{-n}-4p \frac{M-1}{M-p}\left(1-\left(\frac pM\right)^n\right)+\frac{(M-1)p^2}{M-p^2}\left(1-\left(\frac {p^2}M\right)^n\right)\\
   =3-4p \frac{M-1}{M-p}+p^2\frac{M-1}{M-p^2}.
   \end{align*}
Again the rescaled length $\osV_1(K^{(1)}\cap K^{(2)})$ is constant, while the rescaled Euler characteristic of $K^{(1)}\cap K^{(2)}$ depends on $p$ and $M$. Figure~\ref{fig:KcapK} (right) shows plots of $\osV_0(K^{(1)}\cap K^{(2)})$ as a function of $p$ for different parameters $M$. It is apparent that these are positive and monotone  decreasing functions in $p$ for any $M$ and the limit as $M\to\infty$ is given by
$f(p)=3-4p+p^2$. From the existence of the limits $\osV_k(K^{(1)}\cap K^{(2)})$ it is clear that $D'$ as chosen above is the correct scaling exponent. The notation for the limit is justified by the fact that $D'$ is almost surely the Hausdorff dimension of $K^{(1)}\cap K^{(2)}$, as the following statement clarifies.
\end{rem}
\begin{figure}[t]
  \begin{center}
  \includegraphics[width=\textwidth]{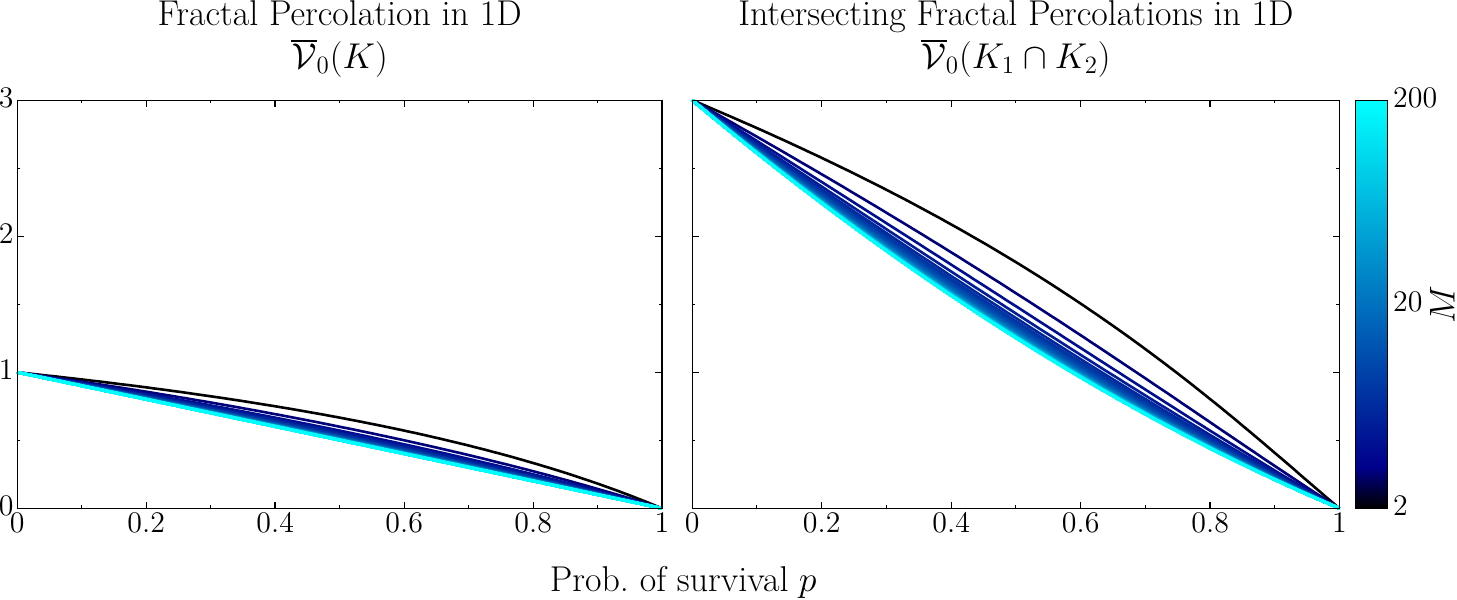}
  \caption{\label{fig:KcapK} The rescaled limits $\osV_0(K)$ (left) and $\osV_0(K^{(1)}\cap K^{(2)})$ (right)
  %of the Euler characteristics of $K_n$ and $K^{(1)}_n\cap K^{(2)}_n$
  as functions of {$p\in(0,1]$} for different values of $M$ (color coded) as given by Corollary~\ref{cor:Vk-dim1} and Remark~\ref{rem:1dim-Klim}, respectively. The limit curves as $M\to\infty$ are shown in red.}
  \end{center}
\end{figure}

\begin{prop} \label{prop:dim_K1capK2}
  Let $K^{(1)}, K^{(2)}$ be independent fractal percolations on $[0,1]$ with the same parameters $M\in\N_{\geq 2}$ and {$p\in(0,1]$}. If $p\leq 1/\sqrt{M}$, then the set $K^{(1)}\cap K^{(2)}$ is almost surely empty. If  $p> 1/\sqrt{M}$, there is a positive probability that $K^{(1)}\cap K^{(2)}\neq \emptyset$ and, conditioned on $K^{(1)}\cap K^{(2)}\neq \emptyset$, we have $\dim_H(K^{(1)}\cap K^{(2)})=D'$ almost surely.
\end{prop}
\begin{proof}
        For any {$p\in(0,1]$}, the set $K^{(1)}\cap K^{(2)}$ can be coupled with a fractal percolation $F$ on $[0,1]$ with parameter $p^2$ (and the same $M$) by retaining an interval $I_\sigma$ of level $n$ if and only if it is contained in both sets $K_n^{(1)}$ and $K_n^{(2)}$. Then $K^{(1)}\cap K^{(2)}$ dominates $F$. Hence, almost surely, $\dim_H\left(K^{(1)}\cap K^{(2)}\right)\geq \dim_H F$. Now observe that conditioning on the event $\{K^{(1)}\cap K^{(2)}\neq\emptyset\}$ is the same as conditioning on $\{F\neq\emptyset\}$. Indeed, on the one hand the first event is obviously satisfied whenever the latter is. On the other hand, if $\{F=\emptyset\}$ holds, then there is some $n\in\N$ such that $F_n=\emptyset$. This implies that, for any $m\geq n$, $K_m^{(1)}\cap K_m^{(2)}$ consists of finitely many isolated points contained in the set $\{\frac k{M^{-n}}:k\in\{1,\ldots,M^n-1\}\}$ (cf.\ the proof of Proposition~\ref{prop:MPcapMP}). In particular, there are no new points generated after the $n$-th step. Each point at level $m\geq n$ is independently retained in the next step with probability $p^2$. This means in particular that $K^{(1)}\cap K^{(2)}$ is empty almost surely under the condition $F=\emptyset$.
        We conclude that, for $p\leq\sqrt{M}$, the set $K^{(1)}\cap K^{(2)}$ is empty almost surely, since $F$ has this property. Moreover, since conditioned on $F\neq\emptyset$ we have $\dim_H F=D'$ almost surely for any $p\geq\sqrt{M}$, we infer from the above inequality that conditioned  $K^{(1)}\cap K^{(2)}\neq\emptyset$, $D'$ is almost surely a lower bound for $\dim_H\left(K^{(1)}\cap K^{(2)}\right)$. (The same is true for the Minkowski dimension.) %It is easy to see that $K^{(1)}\cap K^{(2)}$ is almost surely empty for $p<\sqrt{M}$ and that $\dim_H\left(K^{(1)}\cap K^{(2)}\right)=D'$ for $p\geq\sqrt{M}$.

	We show that $D'$ is also an upper bound for $\dim_H\left(K^{(1)}\cap K^{(2)}\right)$. For any realization of $K^{(1)}\cap K^{(2)}$ and any $\delta>0$, a $\delta$-cover of $K^{(1)}\cap K^{(2)}$ is obtained by taking the cubes of level $n$ (for some $n$ large enough that $M^{-n}<\delta$) contained in $F_n$ (which cover $F$) and adding the finitely many singletons $\left\{\frac k{M^{-n}}\right\}$, $k=1,\ldots,M^{-n}-1$ which clearly cover the additional isolated points in $K^{(1)}\cap K^{(2)}$ not already covered by the chosen intervals. Using these covers and noting that the singletons have diameter zero and the intervals diameter $M^{-n}$, we get for any $s>0$, that $\Ha^s_\delta(K^{(1)}\cap K^{(2)})\leq Z_n M^{-ns}$, where $Z_n$ is the number of cubes in $F_n$. Since $Z_n$ is {the} size of the  $n$-th generation of a Galton-Watson process in which the expected number of offspring of an individuum is $M^{D'}=M^2p$, it is well known that $Z_n M^{-nD'}\to 1$ almost surely as $n\to\infty$. This shows $\Ha^{D'}(K^{(1)}\cap K^{(2)})<\infty$ almost surely and thus $\dim_H\left(K^{(1)}\cap K^{(2)}\right)\leq D'$.
\end{proof}

\paragraph*{The case $d=2$.} %For Mandelbrot percolation
Now we provide proofs of the formulas for the three limit functionals $\osV_k(F)$, $k=0,1,2$ for fractal percolation $F$ in $\R^2$ stated in Theorem~\ref{thm:Vk-limit-dim2}. The starting point is again the general formula in Theorem~\ref{thm:Vk-limit-general}, which can be simplified further by using on the one hand the various symmetries in the fractal percolation model and on the other hand the properties of the functionals.

  \begin{proof}[Proof of Theorem~\ref{thm:Vk-limit-dim2}]
Let $M\in\N_{\geq 2}$ and {$p\in(0,1]$} and $k\in\{0,1,2\}$. By \eqref{eq:Vk-limit-general} in  Theorem~\ref{thm:Vk-limit-general}, we have
 \begin{align} \label{eq:Vk-limit-d=2}
  \osV_k(F)=
    q_{2,k} +\sum_{T\subset\{1,\ldots,M^2\},|T|\geq 2}(-1)^{|T|-1} \sum_{n=1}^\infty r^{n(D-k)} \E V_k(\bigcap_{j\in T} F^j_{n}).
  \end{align}
 Observe that among the intersections $\bigcap_{j\in T} F_n^j$ occurring in \eqref{eq:Vk-limit-d=2} only those need to be considered for which the corresponding intersection $\bigcap_{j\in T} J_j$ of the subcubes $J_j=\phi_j(J)$ is nonempty. All other intersections are empty almost surely and hence their expected intrinsic volumes are zero. The nonempty intersections of subcubes can be reduced to four basic cases, see Figure~\ref{fig:J_j}:
\begin{figure}[t]
     \begin{center}
     \includegraphics[width=3cm]{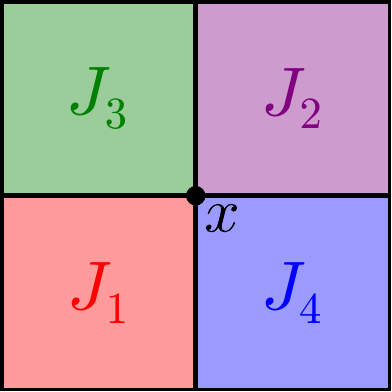}
     \caption{ \label{fig:J_j} Possible mutual positions of the basic cubes $J_j$ which produce nonempty intersections.}
     \end{center}
\end{figure}
There are only two ways in which two subcubes can have a nonempty intersection, namely they can intersect in a common face (like $J_1$ and $J_4$ in Fig,~\ref{fig:J_j}) or in a common corner (like $J_1$ and $J_2$). Three subcubes can only have a nonempty intersection at a common corner (like $J_1$, $J_2$ and $J_3$) and similarly four subcubes can only intersect in a common corner (like $J_1$, $J_2$, $ J_3$ and $J_4$). Only the number of intersections of each of these four types changes with $M$. These numbers are given by $2M(M-1)$, $2(M-1)^2$, $4(M-1)^2$ and $(M-1)^2$, respectively, independent of $p$ and $n$.
Hence formula \eqref{eq:Vk-limit-d=2} reduces to
  \begin{align}
    \osV_k(F) %\lim_{n\to\infty} \overline{v}_k(n)
    =&q_{2,k}
    -2M(M-1)\sum_{n=1}^\infty  r^{n(D-k)} \E V_k(F_n^1\cap F_n^4)\label{eq:Zk-limit-dim2-1}\\
    &-2(M-1)^2 \sum_{n=1}^\infty r^{n(D-k)} \E V_k(F_n^1\cap F_n^2) \notag \\
     &+4(M-1)^2 \sum_{n=1}^\infty r^{n(D-k)} \E V_k(F_n^1\cap F_n^2\cap F_n^3) \notag \\
      &-(M-1)^2 \sum_{n=1}^\infty r^{n(D-k)} \E V_k\left(\bigcap_{j=1}^4 F_n^j\right)%\cap F_n^2\cap F_n^3\cap F_n^4)
      .\notag
  \end{align}

For $k=2$, i.e.\ for the area $V_2$ in $\R^2$, it is enough to observe that the area of all the intersections of the level sets $F_n^j$ in this formula are almost surely zero (they are all contained in a line segment), implying that $\E V_2(\bigcap_{j\in T} F_n^j)=0$ for all $n\in\N$ and all index sets $T$ with $|T|\geq 2$. Therefore,
$$
\osV_2(F)=q_{2,2}=V_2([0,1]^2)=1,
$$
independent of $M$ and $p$ as asserted in Theorem~\ref{thm:Vk-limit-dim2}.

For $k=1$, i.e.\ for the ``boundary length'' $V_1$, only the intersections of the first type $F_n^1\cap F_n^4$ need to be considered, while for the other three types the intersection is at most one point, implying that the expected boundary length vanishes independent of $n$. This yields
\begin{align} \label{eq:V1F}
    \osV_1(F) %\lim_{n\to\infty} Z_1(n)
    &=q_{2,1}
    -2M(M-1)\sum_{n=1}^\infty  r^{n(D-1)} \E V_1(F_n^1\cap F_n^4).
 \end{align}
%and it remains to compute the expectations $\E V_1(F_n^1\cap F_n^2)$, $n\in\N$.
% , which clearly depend on the parameters $p\in[0,1]$ and $M\geq 2$.
We claim that, for each $n\in\N$,
\begin{align} \label{eq:dim-red-k1}
  \E V_1(F_n^1\cap F_n^4)= p^{2n}/M.
\end{align}
We will show below that this follows from Proposition~\ref{prop:MPcapMP}.
Plugging \eqref{eq:dim-red-k1} into equation \eqref{eq:V1F} and recalling that $r^{D-1}=M^{-D+1}=(M\,p)^{-1}$, we conclude
 \begin{align*}
    \osV_1(F) %\lim_{n\to\infty} Z_1(n)&=q_{2,1}
    %-2M(M-1)\sum_{n=1}^\infty  r^{n(D-1)} \E V_1(F_n^1\cap F_n^2)\\
    &=2 - 2(M-1)\sum_{n=1}^\infty (M\,p)^{-n} p^{2n}
    =2-2(M-1)\sum_{n=1}^\infty (p/M)^n\\
    &=2-2(M-1)\frac p{M-p}=\frac{2M(1-p)}{M-p}.
 \end{align*}

For $k=0$, i.e.\ for the Euler characteristic $V_0$, all terms in the above formula \eqref{eq:Zk-limit-dim2-1} are relevant and contribute to the limit. It is rather easy to see that
$V_0(F_n^1\cap F_n^2\cap F_n^3\cap F_n^4)=1$ with probability $p^{4n}$, since at all levels $m=1,\ldots,n$, in each of the four cubes $J_i$, $i=1,\ldots, 4$ the subcube of level $m$ which intersects the common corner needs to survive (which happens with probability $p$,
independently of all the other subcubes of any level). Otherwise the intersection of the four sets $F^j_n$ will be empty. Hence, for each $n\in\N$ (and each $M\geq 2$),
\begin{align}
  \label{eq:type4-0} \E V_0(F_n^1\cap F_n^2\cap F_n^3\cap F_n^4)=p^{4n}.
\end{align}
Therefore, the sum in the last line of formula \eqref{eq:Zk-limit-dim2-1} is given by
\begin{align} \label{eq:E4}
   \sum_{n=1}^\infty r^{nD} \E V_0\left(\bigcap_{j=1}^4 F_n^j\right)= \sum_{n=1}^\infty (r^D\, p^4 )^{n}=\frac{r^D\,p^4}{1-r^D\,p^4}=\frac{p^3}{M^2-p^3},
\end{align}
where the last equality is due to the relation $p\, r^D=r^2=M^{-2}$. (Note that the geometric series above converges, since $p^4 \,r^D=p^3\,M^{-2}<1$ for any {$p\in(0,1]$} and any integer $M\geq2$.)

Similarly, one observes that $V_0(F_n^1\cap F_n^2\cap F_n^3)=1$ with probability $p^{3n}$ and $V_0(F_n^1\cap F_n^2)=1$ with probability $p^{2n}$ for $n\in\N$, which yields for the sums in the third and the second line in formula \eqref{eq:Zk-limit-dim2-1} the expressions
\begin{align}
   \sum_{n=1}^\infty r^{nD} \E V_0(F_n^1\cap F_n^2\cap F_n^3)= \sum_{n=1}^\infty (r^D\,p^3)^{n}=\frac{p^2}{M^2-p^2}
\end{align}
and
\begin{align}
   \sum_{n=1}^\infty r^{nD} \E V_0(F_n^1\cap F_n^2)= \sum_{n=1}^\infty (r^D p^2)^{n}=\frac{p}{M^2-p}.
\end{align}
It remains to compute the expected Euler characteristic for the type $F_n^1\cap F_n^4$.
We claim that, for any $n\in\N$,
\begin{align} \label{eq:dim-red-k0}
   \E &V_0(F_n^1\cap F_n^4)=
    (Mp^2)^n\times\\&\left(\frac 3M-2 M^{-n}
-4 \frac{M-1}{M-p}\left[\frac pM-\left(\frac pM\right)^{n}\right]
+\frac{M-1}{M-p^2}\left[\frac{p^2}M-\left(\frac {p^2}M\right)^{n}\right]\right).\notag
\end{align}
We will demonstrate below that this follows from Proposition~\ref{prop:MPcapMP}. Plugging \eqref{eq:E4} -- \eqref{eq:dim-red-k0} into %the first line of equation
 \eqref{eq:Zk-limit-dim2-1} and computing the remaining series yields the missing terms of $\osV_0(F)$. More precisely, we get for the last sum on the first line of equation \eqref{eq:Zk-limit-dim2-1} the expression
 \begin{align*}
    % \osV_0(F)&=1- E_1 +% - 2 - 2M(M-1) \sum_{n=1}^\infty  r^{n(D-1)} \E V_1(F_n^1\cap F_n^2)=\frac{2(1-p)}{1-p/M},\\
%    (M-1)^2\left(-\frac{2p}{M^2-p}+\frac{4p^2}{M^2-p^2}-\frac{p^3}{M^2-p^3}\right)\\
    E_1&:=\frac{2(M-1)^2p}{M-p}\left(\frac 3{M-1}-\frac{4p}{M-p}+\frac{p^2}{M-p^2}\right)-2M(M-1)^2p\times\\
    &\quad\times\left(\frac{2}{(M-1)(M^2-p)}-\frac{4p}{(M-p)(M^2-p^2)}+\frac{p^2}{(M-p^2)(M^2-p^3)}\right)
  \end{align*}
  and therefore
 \begin{align*}
    \osV_0(F) %=\lim_{n\to\infty} Z_0(n)
    &=1-E_1+(M-1)^2\left(-\frac{2p}{M^2-p}+\frac{4p^2}{M^2-p^2}-\frac{p^3}{M^2-p^3}\right).
\end{align*}
Combining some of the terms gives the formula stated in Theorem~\ref{thm:Vk-limit-dim2} for $\osV_0(F)$.

To complete the proof, it remains to verify equations \eqref{eq:dim-red-k1} and \eqref{eq:dim-red-k0}.
To understand the structure of $F_n^1\cap F_n^4$, it is enough to study the intersection of two independent 1-dimensional fractal percolations $K^{(1)}$ and $K^{(2)}$ defined on a common interval $[0,1]$ (with the same parameters $M$ and $p$ as $F$). %So fix $M$ and $p\in[0,1]$ and
For $n\in \N$ and $i=1,2$, let $K_n^{(i)}$ denote the $n$-th steps of their construction. Similarly as in \eqref{eq:basic-sim}, let $\widetilde K_n^{(i)}$, $i=1,2$ be the random set, which equals $K_n^{(i)}$ with probability $p$ and is empty otherwise, i.e.\ we add an additional $0$-th step to decide whether the set $K_n^{(i)}$, $n\in\N$, is kept or discarded. This is to account for the first step of the construction of $F$ (in which the cubes $J_i$ are discarded with probability $1-p$).
Then, for each $n$, we have the following equality in distribution
\begin{align}
  \label{eq:MpcapMP-dim1-2}
  F_n^1\cap F_n^4=\psi (\widetilde K_{n-1}^{(1)}\cap \widetilde K_{n-1}^{(2)}),
\end{align}
where $\psi:\R\to\R^2, x\mapsto (t/M) a+ ((1-t)/M) b$ is the similarity, which maps $[0,1]$ to the segment $J_1\cap J_4$ with endpoints $a$ and $b$. Since intrinsic volumes are independent of the ambient space dimension, motion invariant and homogeneous, this implies in particular
\begin{align} \label{eq:psi}
   \E V_k(F_n^1\cap F_n^4)&=\E V_k(\psi (\widetilde K_{n-1}^{(1)}\cap \widetilde K_{n-1}^{(2)}))\\
   &=r^k p^2 \E V_k(K_{n-1}^{(1)}\cap K_{n-1}^{(2)}).\notag
\end{align}
Now the claims \eqref{eq:dim-red-k1} and \eqref{eq:dim-red-k0} follow by combining \eqref{eq:psi} with Proposition~\ref{prop:MPcapMP}.
\end{proof}

\begin{rem} \label{rem:speed}
  From the proof of Theorem~\ref{thm:Vk-limit-dim2}, we also get explicit expressions for the expected intrinsic volumes of the approximation sets $F_{n}$ for each $n\in\N$. %Let $\overline{v}_k(n):=r^{n(D-k)}\E V_k(F_{n})$.
  To determine $\overline{v}_k(m):=r^{m(D-k)}\E V_k(F_{m})$, it is enough to truncate all the sums in formula \eqref{eq:Zk-limit-dim2-1} after the $m$-th step and compute the resulting finite geometric sums. This yields for $k=0$ and $n\in\N$,
  \begin{align*}
  \overline{v}_0(n)&=1-\frac{2p(M-1)^2}{M-p}\left(\frac 3{M-1}-\frac{4p}{M-p}+\frac{p^2}{M-p^2}\right)\left[1-\left(\frac pM\right)^n\right]\\
        &+\frac{2p(M^2-1)}{M^2-p}\left[1-\left(\frac{p}{M^2}\right)^n\right]-\frac{4p^2(M-1)^2}{(M-p)^2}\left[1-\left(\frac{p^2}{M^2}\right)^n\right]\\
        &+\frac{p^3(M-1)^2(M+p^2)}{(M-p^2)(M^2-p^3)}\left[1-\left(\frac{p^3}{M^2}\right)^n\right].
\end{align*}
It is easy to see that this sequence converges exponentially fast to $\osV_0(F)$ as $n\to\infty$. More precisely, we have
$$
\overline{v}_0(n)-\osV_0(F)\sim c\,(p/M)^n, \quad \text{ as } n\to \infty,
$$
(i.e.\ the quotient of the left and the right hand side converges to $1$) with the constant $c:=\frac{2p(M-1)^2}{M-p}\left(\frac 3{M-1}-\frac{4p}{M-p}+\frac{p^2}{M-p^2}\right)$ being positive for each {$p\in(0,1]$} and $M\in\N_{\geq 2}$.
  Moreover, the sequence $(\overline{v}_0(n))_n$ is eventually strictly decreasing, i.e.\ strictly decreasing from some index $n_0\in \N$. This exemplifies that the convergence $\overline{v}_k(n)\to \osV_k(F)$ is extremely fast and that the functionals $\osV_k(F)$ can be approximated well by the $\overline{v}_k(n)$. This was also observed in simulations, where already for small $n$ (like $n=8$, even for $M=2$, see Fig.~\ref{fig:V0}) $\overline{v}_k(n)$ is virtually indistinguishable from the limit $\osV_k(F)$, see also Remark~\ref{rem:simu} below. Fast convergence can also be expected for the limits $\osV_k(K)$ of other random self-similar sets $K$, for which no exact formula may be available. It is another intriguing question whether a similar speed of convergence can be expected for the percolation probabilities of $F_n$.
\end{rem}

\begin{rem}(On the simulation study) \label{rem:simu}
   Due to the fast convergence of the studied geometric functionals, their
numerical estimation is efficient and accurate.
A simulation study demonstrates their potential as robust shape descriptors for
applications, see Fig.~\ref{fig:V0}.
To generate the approximations of fractal percolation, we create black-and-white pixel images by hierarchically simulating the survival or death of squares (given by patches of pixels).
We use the MT19937 generator \cite{MatsumotoNishimura1998} (known as
``Mersenne Twister'') to generate the required Bernoulli variables.
%as implemented in the \textsc{GNU Scientific Library}.
Taking advantage of the additivity of the Minkowski functionals, we
compute the Euler characteristic using an efficient algorithm, where the
computation time grows linearly with the system size.
We simply iterate over all $2\times 2$ neighborhoods of pixels and add the
corresponding values from a look-up table as described
in~\cite{GKSM2013}.
In two separate simulations using analogous parameters, we have computed
the Euler characteristic of $F_n$ and $C_n$ (see Section~\ref{sec:5}).

We simulate realizations of finite approximations for $M=2$, $3$, and $4$.
For each value of $M$, we choose three levels $n$ of the approximation
$n=32/(2^M)$, $32/(2^M)+2$, or $32/(2^M)+4$.
Since the rate of convergence increases with $M$, for larger $M$
smaller values of $n$ are sufficient. % for a good approximation.
For each chosen value of the probability of survival $p=0.11, 0.13, \ldots, 0.99$,
we simulate $75000$, $5000$, or $2500$ samples for $M=2$, $3$, or $4$, respectively.
Only in the case of $M=2, p\leq 0.31$ for $F_n$, the number of samples
is increased by a factor 10 for improved statistics.

The mean values are unbiasedly estimated by the arithmetic mean of the
Euler characteristic of the samples.
The error bars in the plots represent %the standard error of the mean, which is estimated by
the sample standard deviations.
The simulation results, shown in Figure~\ref{fig:V0},
are in excellent agreement with the analytic curves, see Remarks~\ref{rem:speed} and \ref{rem:speed2}.
The code is freely available via GitHub \cite{KWgithub}.
\end{rem}

\begin{rem} \label{rem:K1capK2}
{An essential observation in the proof of
  Theorem~\ref{thm:Vk-limit-dim2} (cf.\ eq.~\eqref{eq:MpcapMP-dim1-2}
  and the discussion preceding it) is that any intersection
  $F^{(1)}\cap F^{(2)}$ of two fractal percolations constructed in
  neighboring squares sharing a common side, can be modelled by the
  intersection $K^{(1)}\cap K^{(2)}$ of two independent
  one-dimensional fractal percolations $K^{(1)}, K^{(2)}$ on that side (with the same
  parameters $M$ and $p$ as the $F^{(i)}$).
  %It is clear from equation \eqref{eq:MpcapMP-dim1-2} that the
  %intrinsic volumes of the sets $K^{(1)}_n\cap K^{(2)_n}$ and
  %$F^{(1)}_n\cap F^{(2)}_n$ coincide for each $n$. Furthermore,
  More precisely, the random sets $F^{(1)}\cap F^{(2)}$ and $K^{(1)}\cap
  K^{(2)}$ are equal in distribution.
  Therefore their Hausdorff and Minkowski dimensions must coincide.
  % of $F^{(1)}\cap F^{(2)}$ coincides with that of $K^{(1)}\cap
  % K^{(2)}$,
  The almost sure dimension of the latter set has been determined in
  Proposition~\ref{prop:dim_K1capK2}. Moreover,
  Proposition~\ref{prop:dim_K1capK2} states that the intersection
  $K^{(1)}\cap K^{(2)}$ is almost surely empty for any $p\leq
  1/\sqrt{M}$ and so the same must hold for $F^{(1)}\cap F^{(2)}$. This
  observation allows a short alternative proof of the lower bound
  $1/\sqrt{M}$ of Chayes, Chayes and Durrett \cite{CCD88} for the
  percolation threshold of fractal percolation $F$ in $[0,1]^2$ (see
  \eqref{eq:pc-bounds}): The intersection of $F$ with any vertical line
  of the form $y=
  %segment $[(x,0), (x,1)]$ in the unit square, where $x$ is of the form $
  k/M^n$, where $n\in\N$ and $k\in\{1,\ldots, M^n-1\}$, can be modeled
  as a union of $M^n$ small copies of $K^{(1)}\cap K^{(2)}$. Any path in
  $F$ from left to right need to pass this line, which is impossible if
  these intersections are empty a.s., i.e. for any $p\leq 1/\sqrt{M}$.}
\end{rem}

\begin{rem}
  It is easy to see from Theorem~\ref{thm:Vk-limit-general} that also for fractal percolation in $\R^d$, the rescaled limit $\osV_d(F)$ of the volume equals 1 for any $p$ and $M$. Indeed, none of the intersections occurring in formula \eqref{eq:Vk-limit-general} will contribute to the limit as they are contained in lower dimensional subsets of $\R^d$.
\end{rem}

\section{Approximation of $F$ by the closed complements of $(F_{n})_n$.} \label{sec:5}
Now we consider the closed complements $C_{n}:=\cl{J\setminus F_n}$, $n\in\N_0$, of the construction steps $F_n$ of the fractal percolation process inside the unit cube $J=[0,1]^d$. Note that $C_0=\emptyset$, since $F_0=J$. The random sets $C_{n}$ are also given by
$$
C_{n}=\bigcup_{\sigma\in\Sigma_n\setminus\sT_n} J_\sigma,
$$
cf.\ Section~\ref{sec:2}, implying in particular that each realization of $C_{n}$ consists of a finite number of closed cubes and is thus polyconvex. Hence intrinsic volumes are well defined. The set $C_{n}$ consists of those subcubes $J_\sigma$ of level $n$ for which at least one of the cubes $J_{\sigma|i}$, $i\in\{1,\ldots, n\}$, was discarded.
We also introduce, for each $j\in\{1,\ldots,M^d\}$ (and each $n\in\N_0$), the set
$$
C_{n}^j:=\bigcup_{{\sigma\in\Sigma_n\setminus\sT_n},{\sigma|1=j}} J_\sigma,
$$
as the union of those cubes of level $n$ which are contained in $J_j\cap C_{n}$.

We are interested in the expected intrinsic volumes $\E V_k(C_{n})$, $k=0,\ldots, d$ and in particular in the limiting behaviour as $n\to\infty$, for which we have the following general formula analogous to \eqref{eq:Vk-limit-general} in Theorem~\ref{thm:Vk-limit-general}.

\begin{thm} \label{thm:Zk-limit-general2}
{Let $k\in\{0,\ldots,d-1\}$ and let $F$ be a fractal percolation in $\R^d$ with parameters $M\in\N_{\geq 2}$ and $p\in(r^{d-k},1]$. Let $D$ be the Minkowski dimension of $F$, see \eqref{eq:dimF}.  Then, the limit}
  $$
  %\lim_{n\to\infty} \overline{v}_k(n)=
  \osV^c_k(F):=\lim_{n\to\infty} r^{n(D-k)} \E V_k(C_{n})
  $$
  exists and is given by the expression
  \begin{align} \label{eq:Zk-limit-general2}
    q_{d,k}\frac{M^{d-k}(1-p)}{M^{d-k}p-1} +\sum_{T\subset\{1,\ldots,M^d\},|T|\geq 2}(-1)^{|T|-1} \sum_{n=1}^\infty r^{n(D-k)} \E V_k(\bigcap_{j\in T} C_{n}^j),
  \end{align}
  where as before $q_{d,k}=V_k([0,1]^d)$.
\end{thm}

\begin{rem}
{The condition $p>r^{d-k}$, which is equivalent to $k<D$, is a
natural restriction for the existence of $\osV^c_k(F)$. If the
dimension $D$ of $F$ is smaller than the homogeneity index $k$ of the
functional, then the `edge effects' caused by the common boundary of
$C_{n}$ with $J=[0,1]^d$ will dominate the limiting behaviour and
therefore a different rescaling will be necessary. More precisely, since
for the cube $J$ no rescaling is necessary for the intrinsic volumes,
one would expect the limit $\lim_{n\to\infty} r^{n(k-k)} \E V_k(C_{n})$
to converge instead which is too rough to see the lower-dimensional set $F$. For $D<d-1$, for instance, it is easy to see that the surface area $C_{d-1}(C_{n},\bd J)\to V_{d-1}(J)$ as $n\to\infty$, while $C_{d-1}(C_{n},\bd F_n)\approx r^{(d-1-D)n}\to 0$.  We refer also to Corollary \ref{cor:5.2} below, where we compute explicitly  $\osV^c_0(F)$  for fractal percolation $F=F_p$ on the unit interval for all parameters $p$. It turns out that, for $p<1/M$, the edge effects dominate and we have $\osV^c_0(F_p)=\infty$, while $\osV^c_0(F_p)$ is finite for all $p\geq 1/M$, including the critical case $p=1/M$, for which the `edge effects' matter and contribute a second term in the limit. We expect that this is the generic behaviour of all functionals $\osV^c_k(F)$ in any dimension: convergence at and divergence below their critical value $p=r^{d-k}$.
}

\end{rem}
\begin{proof}
We follow the lines of the proof of Theorem~\ref{thm:Vk-limit-general}. Let
$$
\overline{v}^c_k(n):=r^{n(D-k)} \E V_k(C_{n}), \quad n\in\N_0.
$$
Since $C_0=\emptyset$, we have $\overline{v}^c_k(0)=\E V_k(\emptyset)=0$.
Setting
$$
w_k(n):= \overline{v}^c_k(n)-\overline{v}^c_k(n-1), \quad n\in\N,
$$
we observe that, similarly as in \eqref{eq:Z-w-relation} above,
\begin{align} \label{eq:Z-w-relation2}
\osV^c_k(F)=\lim_{n\to\infty} \overline{v}^c_k(n)=\sum_{n=1}^\infty w_k(n).
\end{align}
By definition of $w_k$, we have
\begin{align} \label{eq:Z-w-relation3}
  w_k(n)&= \overline{v}^c_k(n)-\overline{v}^c_k(n-1)%= r^{n(D-k)} \E V_k(C_{n})- r^{(n-1)(D-k)} \E V_k(C(n-1))\\
 % &= r^{n(D-k)}\left( \E V_k(F_n)- r^{k-D} \E V_k(F_{n-1})\right)\\
  = r^{n(D-k)}\left( \E V_k(C_{n})- M^dp \, r^{k} \E V_k(C_{n-1})\right).
 % &= r^{n(D-k)}\left( \E V_k(F_n)- M^dp \E V_k(F_n^1)\right),\\
%  &= r^{n(D-k)}\left( \E V_k(F_n)- \sum_{j} \E V_k(F_n^j)\right),\\
\end{align}
%where the have used that $M^dp=r^{-D}$.
Now recall from \eqref{eq:basic-sim} that, for each $j\in\{1,\ldots,M^d\}$, $F_n^j$ survives the first construction step with probability $p$ (in which case it is distributed like $\phi_j(F_{n-1})$) and it is empty otherwise. Thus we get for the closed complements $C_{n}^j$
\begin{align*}
  \E V_k(C_{n}^j)&=p \E V_k(\phi_j(C_{n-1}))+(1-p) \E V_k(J_j)\\
  &=p r^k \E V_k(C_{n-1})+(1-p)r^k q_{d,k},
\end{align*}
and therefore
$$
\sum_{j=1}^{M^d} \E V_k(C_{n}^j)=M^d p r^k \E V_k(C_{n-1})+M^d(1-p)r^k q_{d,k}. %\right)%=r^{k-d}\left(p \E V_k(C_{n-1})+(1-p) q_{d,k}\right).
$$
Plugging this into \eqref{eq:Z-w-relation3} and recalling that $r^{-D}=M^dp$
 yields
\begin{align*}
  w_k(n)&= r^{n(D-k)}\left( \E V_k(C_{n})- \sum_{j=1}^{M^d} \E V_k(C_{n}^j)\right)+r^{(n-1)(D-k)}\frac{1-p}p q_{d,k}.
 % &= r^{n(D-k)}\left( \E V_k(F_n)- M^dp \E V_k(F_n^1)\right),\\
%  &= r^{n(D-k)}\left( \E V_k(F_n)- \sum_{j} \E V_k(F_n^j)\right),\\
\end{align*}
Since $C_{n}=\bigcup_{j=1}^{M^d} C_{n}^j$, by the inclusion-exclusion principle, this can be expressed in a more convenient form:
%$$
%\E V_k(C_{n})-\sum_{j=1}^{N} \E V_k(C_{n}^j)=\sum_{T\subset\{1,\ldots,M^d\},|T|\geq 2}(-1)^{|T|-1} \E V_k(\bigcap_{j\in T} C_{n}^j),
%$$
%which yields
for each $n\in\N$ and each $k\in\{0,\ldots,d\}$,
\begin{align}\label{eq:wk-rep2}
  w_k(n)&=  r^{(n-1)(D-k)}\frac{1-p}p q_{d,k}+\sum_{\begin{array}{@{}c@{}}
    {\scriptstyle T\subset\{1,...,M^d\}}\\[-1mm]{\scriptstyle |T|\geq 2}
  \end{array}}(-1)^{|T|-1} r^{n(D-k)} \E V_k\Big(\bigcap_{j\in T} C_{n}^j\Big).
\end{align}
Note that this is again a finite sum with a fixed number of terms (independent of $n$) and that all the intersections appearing in this formula are at most $d-1$-dimensional. The summation over $n$ can be shown to converge for each summand separately (see Proposition~\ref{prop:Rk-conv2} below; this is where the hypothesis {$p>r^{d-k}$ is used}).

 Inserting the representation \eqref{eq:wk-rep2} for $w_k$ into equation \eqref{eq:Z-w-relation2} yields
  \begin{align} \label{eq:Yk-proof}
    \osV^c_k(F) \notag&=\sum_{n=1}^\infty \left(r^{(n-1)(D-k)}\frac{1-p}p q_{d,k}+\sum_{\begin{array}{@{}c@{}}
    {\scriptstyle T\subset\{1,...,M^d\}}\\[-1mm]{\scriptstyle |T|\geq 2}
  \end{array}}\hspace{-3pt}(-1)^{|T|-1} r^{n(D-k)} \E V_k\Big(\bigcap_{j\in T} C_{n}^j\Big)\right)\notag\\
  %&=\sum_{n=1}^\infty \left(r^{(n-1)(D-k)}\frac{1-p}p q_{d,k}+\sum_{T\subset\{1,\ldots,M^d\},|T|\geq 2}(-1)^{|T|-1} r^{n(D-k)} \E V_k(\bigcap_{j\in T} C_{n}^j)\right)\\
    &=q_{d,k}\frac{1-p}p\sum_{n=0}^\infty r^{n(D-k)}+\sum_{\begin{array}{@{}c@{}}
    {\scriptstyle T\subset\{1,...,M^d\}}\\[-1mm]{\scriptstyle |T|\geq 2}
  \end{array}}(-1)^{|T|-1} \sum_{n=1}^\infty  r^{n(D-k)} \E V_k(\bigcap_{j\in T} C_{n}^j),
  \end{align}
  where the convergence of the geometric series in the first term is  due to the assumption {$p>r^{d-k}$, which implies $D>k$}. {The convergence of the series in the last expression for each index set $T$ is ensured by} Proposition~\ref{prop:Rk-conv2} just below {(for which condition $p>r^{d-k}$ is needed again)}, justifying in particular the interchange of the summations and showing the existence of the limit of the $\overline{v}^c_k(n)$ as $n\to\infty$. Now formula \eqref{eq:Zk-limit-general2} follows easily by computing the series in the first term and recalling that $r^{-D}=M^dp$.
\end{proof}
\begin{prop}\label{prop:Rk-conv2}
{Let $k\in\{0,\ldots,d-1\}$ and $F$ be a fractal percolation in $[0,1]^d$ with parameters $M\in\N_{\geq 2}$ and $p\in(r^{d-k},1]$.
  Then, for each $T\subset\{1,\ldots, M^d\}$ with $|T|\geq 2$,}
  $$
  \sum_{n=1}^\infty r^{n(D-k)} \E C_k^\var(\bigcap_{j\in T} C_{n}^j)<\infty,
  $$
  where{,} as before{,} $C_k^{\var}(K)$ denotes the total mass of the total variation measure of the $k$-th curvature measure of a polyconvex set $K$.
  In particular, the sums
  $$
  \sum_{n=1}^\infty r^{n(D-k)} \E V_k(\bigcap_{j\in T} C_{n}^j)
  $$
  converge absolutely.
\end{prop}
We postpone the proof of Proposition~\ref{prop:Rk-conv2} to the last section.

\paragraph*{The case $d=1$.} In order to derive explicit formulas for the limits $\osV^c_k(F)$ in $\R^2$ it is again necessary to discuss these functionals in $\R$ first. We start with a general formula to determine the intrinsic volumes of a polyconvex set $C\subset\R$ from the intrinsic volumes of its closed complement. This will be used to derive expressions for $\E V_k(D_n^{(1)})$ and $\E V_k(D_n^{(1)}\cap D_n^{(2)})$ from the ones already obtained in Section~\ref{sec:3} for $\E V_k(K_n^{(1)})$ and $\E V_k(K_n^{(1)}\cap K_n^{(2)})$, where $D_n^{(i)}:=\overline{I\setminus K_n^{(i)}}$.
\begin{figure}[t]
      \begin{center}
        \includegraphics[width=5cm]{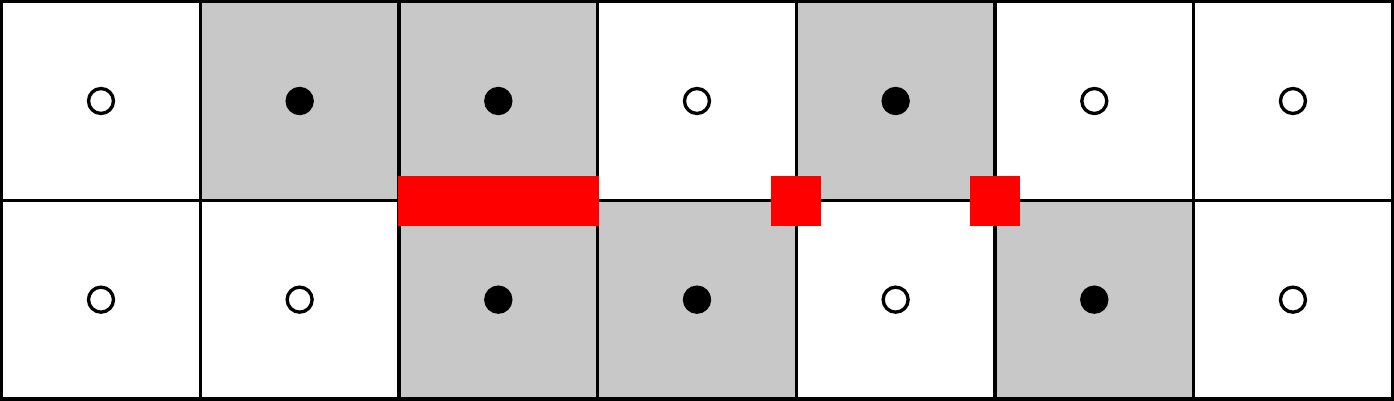}
      \end{center}
       \caption{\label{fig:isolated-points} In the intersection of two unions of intervals (or of unions of cubes in neighboring rows) isolated points appear, which need to be taken into account in the formulas.}
  \end{figure}

  \begin{lem}\label{lem:compl-dim1}
  Let $I:=[0,1]\subset \R$ be the unit interval and let $K\subset I$ be polyconvex (i.e.\ a finite union of intervals). Then the closed complement $C:=\cl{I\setminus K}$ of $K$ within $I$ is polyconvex, $V_1(C)=1-V_1(K)$ and
    \begin{align*}
      V_0(C)=1+V_0(K) -\ind_K(0)-\ind_K(1)-N(K),
    \end{align*}
    where $\ind_A$ denotes the indicator function of a set $A$ and $N(A)$ is the number isolated points in $A$.
      Moreover, if $K'\subset I$ is a second polyconvex set and $C':=\cl{I\setminus K'}$, then
    \begin{align*}
      V_1(C\cap C')&=1-V_1(K)-V_1(K')+V_1(K\cap K') \quad \text{ and }\\
      V_0(C\cap C')&=1+V_0(K)+V_0(K')-V_0(K\cap K') -\ind_{K\cup K'}(0)\\&\phantom{==}-\ind_{K\cup K'}(1)-N(K)-N(K')+N(K\cap K').
    \end{align*}
  \end{lem}
  \begin{proof}
  The first formula is an easy consequence of the additivity of $V_1$ noting that $I=K\cup C$, $V_1(I)=1$ and $V_1(C\cap K)=0$. The second formula for $V_1$ follows from the first one and additivity by noting that
  \begin{align}
    \label{eq:compl-dim1-proof1} C\cup C'=\overline{I\setminus (K\cap K')}.
  \end{align}

    In $\R$ the Euler characteristic $V_0$ equals the number of connected components of a polyconvex set.
    If $K\subset I$ has $k$ connected components, $k\in\N_0$, then $K^c=\R\setminus K$ has $k+1$ (including the two unbounded ones) and so $I\cap K^c$ has $k+1-\ind_K(0)-\ind_K(1)$, since an unbounded connected component of $K^c$ does only contribute a component to $I\cap K^c$ if it has nonempty intersection with $I$, that is, if $0$ or $1$, respectively, are not in $K$. Finally, taking the closure of $I\setminus K$ leaves the number of connected components unchanged, provided there are no isolated points in $K$. Any isolated point, however, reduces the number of connected components in $C$ by one, since it causes the two connected components of $I\setminus K$ adjacent to this point to merge to one component of $C$. This shows the first formula for $V_0$. The second formula follows from the first one taking into account \eqref{eq:compl-dim1-proof1}:
    \begin{align*}
      V_0&(C\cap C')=V_0(C)+V_0(C')-V_0(C\cup C')\\
      &=1+ V_0(K) -\ind_{K}(0)-\ind_{K}(1)-N(K)\\
      &\phantom{==}+1+V_0(K') -\ind_{K'}(0)-\ind_{K'}(1)-N(K')\\
      &\phantom{==}-1-V_0(K\cap K') +\ind_{K\cap K'}(0)+\ind_{K\cap K'}(1)+N(K\cap K')\\
      &=1+V_0(K)+V_0(K')- V_0(K\cap K')-\ind_{K\cup K'}(0)-\ind_{K\cup K'}(1)\\
      &\phantom{==}-N(K)-N(K')+N(K\cap K'),
    \end{align*}
    where we have used the additivity of the indicator function, implying $\ind_{K}+\ind_{K'}=\ind_{K\cap K'}+\ind_{K\cup K'}$. This completes the proof of the last formula.
  \end{proof}

  It is clear that corresponding formulas hold for the expected intrinsic volumes of random polyconvex subsets $K$ and $K'$ of $[0,1]$ and their closed complements. Note that the functional $N$ counting the number of isolated points is not additive. Below we always have the situation that $K$ and $K'$ have no isolated points, %(and neither does $K\cup K'$),
  while isolated points may appear in the intersection $K\cap K'$.

\begin{cor} \label{cor:MPcapMP-compl}
Let $K^{(1)}, K^{(2)}$ be two independent fractal percolations on the interval $I=[0,1]$ both with the same parameters $M\in\N_{\geq 2}$ and {$p\in(0,1]$}. For $n\in\N_0$, let $K_n^{(i)}$ denote the $n$-th step of the construction of  $K_i$, $i=1,2$ and let $D_n^{(i)}:=\cl{I\setminus K_n^{(i)}}$.
Then, for any $n\in\N_0$,
  \begin{align*}
    \E V_1(D_n^{(1)}\cap D_n^{(2)})&=1-2\E V_1(K_n^{(1)})+\E V_1(K_n^{(1)}\cap K_n^{(2)})%=p^{2n}
     \qquad \text{ and }\\
    \E V_0(D_n^{(1)}\cap D_n^{(2)})&=2\E V_0(K_n^{(1)})-\E V_0(K_n^{(1)}\cap K_n^{(2)})+\E N(K_n^{(1)}\cap K_n^{(2)})
    \\&\phantom{=}+1-4p^n+2p^{2n}.%(Mp^2)^n\left(3-2 M^{-n}-4p \frac{M-1}{M-p}\left(1-\left(\frac pM\right)^n\right)\right).
    %(Mp^2)^n\times\\&\left(3-2 M^{-n}
%-4p \frac{M-1}{M-p}\left(1-\left(\frac pM\right)^n\right)
%+\frac{(M-1)p^2}{M-p^2}\left(1-\left(\frac {p^2}M\right)^n\right)\right).
  \end{align*}
  Moreover, we have $\E V_1(D_n^{(1)})=1-\E V_1(K_n^{(1)})$ and
  \begin{align*}
     \E V_0(D_n^{(1)})= \E V_0(K_n^{(1)})+1-2p^n.
  \end{align*}
\end{cor}
\begin{proof}
Applying Lemma~\ref{lem:compl-dim1} to realizations $C$, $C'$ of the random sets $D_n^{(1)}$, $D_n^{(2)}$, respectively, and taking expectations, we obtain, for $k=1$, directly the formula stated above
%\begin{align*}
%  \E V_1(D_n^{(1)}\cap D_n^{(2)})=1-2\E V_1(K_n^{(1)})+\E V_0(K_n^{(1)}\cap K_n^{(2)}),
%\end{align*}
and, for $k=0$,
 \begin{align*}
  \E V_0(D_n^{(1)}\cap D_n^{(2)})&=1+ \E V_0(K_n^{(1)})+\E V_0(K_n^{(2)})-\E V_0(K_n^{(1)}\cap K_n^{(2)})\\
  &\quad -\P(0\in K_n^{(1)}\cup K_n^{(2)})-\P(1\in K_n^{(1)}\cup K_n^{(2)})\\
  &\quad -\E N(K_n^{(1)})-\E N(K_n^{(2)})+\E N(K_n^{(1)}\cap K_n^{(2)}).
 % &=2\E V_0(K_n^{(1)})-\E V_0(K_n^{(1)}\cap K_n^{(2)})-1+2(1-p^n)^2,
\end{align*}
Now observe that almost surely the set $K_n^{(i)}$ contains no isolated points, implying that $\E N(K_n^{(i)})=0$. Moreover,
\begin{align*}
 \P(0\in K_n^{(1)}\cup K_n^{(2)})%&=\P(0\in K_n^{(1)}\cup K_n^{(2)})\\
 &=\P(0\in K_n^{(1)})+\P(0\in K_n^{(2)})-\P(0\in K_n^{(1)}\cap K_n^{(2)})\\&=2p^n-p^{2n},
\end{align*}
and similarly for the point $1$ instead of $0$. This shows the second formula.
The third formula is a direct application of the first formula in Lemma~\ref{lem:compl-dim1} to the realizations $C$ of the random set $D_n^{(1)}$. Similarly, the last formula follows by applying the second formula in Lemma~\ref{lem:compl-dim1} to the realizations $C$ of $D_n^{(1)}$ and taking expectations:
 \begin{align*}
  \E V_0(D_n^{(1)})&=1+\E V_0(K_n^{(1)})-\P(0\in K_n^{(1)})-\P(1\in K_n^{(1)})-\E N(K_n^{(1)})\\
  &=\E V_0(K_n^{(1)})+1 -2p^n,
\end{align*}
since $ \P(0\in K_n^{(1)})=\P(1\in K_n^{(1)})=p^n$ and $\E N(K_n^{(1)})=0$.
\end{proof}
To get more explicit expressions for $\E V_k(D_n^{(1)})$ and $\E V_k(D_n^{(1)}\cap D_n^{(2)})$ from Corollary~\ref{cor:MPcapMP-compl}, we can employ Propositions~\ref{prop:MP} and \ref{prop:MPcapMP} where formulas for $\E V_k(K_n^{(1)})$ and $\E V_k(K_n^{(1)}\cap K_n^{(2)})$ have been derived. The missing piece is an explicit expression for the expected number $\E N(K_n^{(1)}\cap K_n^{(2)})$ of isolated points.

\begin{prop} \label{prop:N-of-MPcapMP}
Let $K^{(1)}, K^{(2)}$ be independent fractal percolations on the interval $[0,1]$ both with the same parameters $M$ and $p$.
Then, for any $n\in\N_0$,
  \begin{align*}
    \E &N(K_n^{(1)}\cap K_n^{(2)})\\
    &=(Mp^2)^n\left(2-2 M^{-n}
-4p \frac{M-1}{M-p}\left[1-\left(\frac pM\right)^n\right]
+2p^2\frac{M-1}{M-p^2}\left[1-\left(\frac {p^2}M\right)^n\right]\right).
  \end{align*}
  \end{prop}
  \begin{proof}
   First observe that for $n=0$ both sides of the formula equal zero and thus the formula holds in this case.
   Let $N(n):=N(K_n^{(1)}\cap K_n^{(2)})$ and let $N_j(n):=N(K_n^{(1)}\cap K_n^{(2)}\cap O_j)$, $j=1,\ldots,M$, be the number of those isolated points contained in the open subinterval $O_j:=\text{int}(J_j)=((j-1)/M,j/M)$. Then obviously
   \begin{align} \label{eq:EN1}
     N(n)=\sum_{j=1}^M N_j(n) +\sum_{j=1}^{M-1} \ind\{j/M \text{ isolated in } K_n^{(1)}\cap K_n^{(2)}\}.
   \end{align}
   Due to the self-similarity, $N_j(n)$ has the same distribution as $\widetilde{N}(n-1)$, where $\widetilde{N}(n-1)$ is the random variable which equals $N(n-1)$ with probability $p^2$ and is zero otherwise (which accounts for the effect that $J_j$ may be discarded in the first construction step of $K^{(1)}$ or $K^{(2)}$, in which case there are no isolated points generated). Moreover, by symmetry, the indicator variables in the second sum all have the same distribution given by
   \begin{align*}
     \P(\{1/M \text{ isolated in } K_n^{(1)}\cap K_n^{(2)}\})=2 p^{2n}(1-p^n)^2=:q_n(p). %2p^{2n}(1-2p^n+p^{2n}).
   \end{align*}
   Indeed, in order for $1/M$ to be isolated, either both $K_n^{(1),1}$ and $K_n^{(2),2}$ (with $K_n^{(i),j}$ as defined in \eqref{eq:Fj_ndef}) need to have a nonempty intersection with $1/M$ (which happens with probability $p^{2n}$) while at the same time both $K_n^{(1),2}$ and $K_n^{(2),1}$ do not intersect $1/M$ (probability $(1-p^n)^2$) or we have the same situation exactly reversed, i.e.\  $K_n^{(1),2}$ and $K_n^{(2),1}$ intersect $1/M$ while $K_n^{(1),1}$ and $K_n^{(2),2}$ do not. %(Here $K_n^{(i),j}$ denotes the union of those cubes of level $n$ contained in $K_n^{(i)}\cap J_j$.)
   Taking expectations in \eqref{eq:EN1}, we get
    \begin{align*}
     \E N(n)&=\sum_{j=1}^M p^2 \E N(n-1) +(M-1) \P(\{1/M \text{ isolated in } K_n^{(1)}\cap K_n^{(2)}\})\\
           &= Mp^2 \E N(n-1)+(M-1)q_n(p),
   \end{align*}
   which is a recursion relation for the sequence $(\gamma_n)_{n\in\N}$ with $\gamma_n:=\E N(n)$.
   By induction, we infer that
   \begin{align*}
      \gamma_n=(M-1)q_n(p)+ Mp^2\gamma_{n-1}%=q_n(p)+ Mp^2 q_{n-1}(p)+ (Mp^2)^2 \gamma_{n-2}
      =\ldots =(M-1)\sum_{s=1}^{n} (Mp^2)^{n-s} q_{s}(p).
   \end{align*}
   Since $q_s(p)=2 p^{2s} (1-2p^s+p^2s)=2 (p^{2s}-2p^{3s}+p^{4s})$, we conclude that
   \begin{align*}
      \gamma_n&=2(M-1)(Mp^2)^{n}\sum_{s=1}^{n} (Mp^2)^{-s} (p^{2s}-2p^{3s}+p^{4s})\\
      &=2(M-1)(Mp^2)^{n}\sum_{s=1}^{n} \left((1/M)^s-2 (p/M)^s+(p^2/M)^s\right)\\
      &= 2(M-1)(Mp^2)^{n}\\
      &\qquad \left(
      \left[1-\left(\frac1M\right)^n\right]-\frac{2p}{M-p}\left[1-\left(\frac pM\right)^n\right] +\frac{p^2}{M-p^2}\left[1-\left(\frac{p^2}M\right)^n\right]
            \right),
   \end{align*}
   from which the expression stated in Proposition~\ref{prop:N-of-MPcapMP} easily follows. %This completes the proof.
   %Isolated points counted in $N$ only occur at intersection points of consecutive intervals of level $n$, i.e.\ at the points $q\in Q_n:=\{k/M^n: k\in\{1,\ldots,M^n-1\}\}$. Observe that $Q_{n}\subset Q_{n+1}$, $n\in\N$, i.e.\ we have an increasing family of sets.  It is helpful to distinguish the points in $Q_n$ according to the level $k$ at which they first appear in a set $Q_k$.
%   Let $N_k$, $k=1,\ldots,n$, be the (random) number of isolated points $z$ in $K_n^{(1)}\cap K_n^{(2)}$ such that $z\in Q_k\setminus Q_{k-1}$. In other words, let
%   $N_k:=\sum_{z\in Q_k\setminus Q_{k-1}} \ind\{z \text{ isolated in } K_n^{(1)}\cap K_n^{(2)}\}$
   \end{proof}
Now we are ready to derive explicit expressions for $\E V_k(D_n^{(1)})$ and $\E V_k(D_n^{(1)}\cap D_n^{(2)})$ from Corollary~\ref{cor:MPcapMP-compl}.
\begin{thm} \label{thm:MPcapMP-compl}
Let $K^{(1)}, K^{(2)}$ be independent fractal percolations on the interval $I=[0,1]$ both with the same parameters $M\in\N_{\geq 2}$ and $p\in(0,1]$. For $n\in\N_0$, let $K_n^{(i)}$ be the $n$-th step of the construction of  $K_i$, $i=1,2$ and let $D_n^{(i)}:=\cl{I\setminus K_n^{(i)}}$.\\
Then, for any $n\in\N_0$, $\E V_1(D_n^{(1)}\cap D_n^{(2)})=1-2p^n+p^{2n}$ and
  \begin{align*}
    \E V_0(D_n^{(1)}\cap D_n^{(2)})=&2 (Mp)^n\left(1-p\frac{M-1}{M-p}\left[1-\left(\frac pM\right)^n\right]\right)+1-4p^n+2p^{2n}\\
    &+(Mp^2)^n\left(-1+p^2\frac{M-1}{M-p^2}\left[1-\left(\frac {p^2}M\right)^n\right]\right).
  \end{align*}
  Moreover, we have $\E V_1(D_n^{(1)})=1-p^n$ and
  \begin{align*}
     \E V_0(D_n^{(1)})= (Mp)^n\left(1-p\frac{M-1}{M-p}\left[1-\left(\frac {p}M\right)^n\right]\right)+1-2p^n.
  \end{align*}
\end{thm}
\begin{proof}
   Combine Corollary~\ref{cor:MPcapMP-compl} with Propositions~\ref{prop:MP}, \ref{prop:MPcapMP}, and \ref{prop:N-of-MPcapMP}.
   The two formulas for $V_1$ and also the one for $\E V_0(D_n^{(1)})$ follow at once. In case of $\E V_0(D_n^{(1)}\cap D_n^{(2)})$ observe that, for any $n\in\N_0$,
  \begin{align*}
\E &N(K_n^{(1)}\cap K_n^{(2)})-\E V_0(K_n^{(1)}\cap K_n^{(2)})\\
    &=(M p^2)^n\left(2-2 M^{-n}
-4p \frac{M-1}{M-p}\left[1-\left(\frac pM\right)^n\right]
+2p^2\frac{M-1}{M-p^2}\left[1-\left(\frac {p^2}M\right)^n\right]\right)\\
%(Mp^2)^n\left(3-2 M^{-n}-4p \frac{M-1}{M-p}\left(1-\left(\frac pM\right)^n\right)\right).
    &\quad -(Mp^2)^n\left(3-2 M^{-n}
-4p \frac{M-1}{M-p}\left[1-\left(\frac pM\right)^n\right]
+p^2\frac{M-1}{M-p^2}\left[1-\left(\frac {p^2}M\right)^n\right]\right)\\
&=(Mp^2)^n\left(-1+ p^2\frac{M-1}{M-p^2}\left[1-\left(\frac {p^2}M\right)^n\right]\right)\\
&=(Mp^2)^n\left(\frac{M(p^2-1)}{M-p^2}-\frac{(M-1)p^2}{M-p^2}\left(\frac {p^2}M\right)^n\right)
  \end{align*}
 and thus $\E V_0(D_n^{(1)}\cap D_n^{(2)})$ equals
 \begin{align*}
   2\E V_0(K_n^{(1)})+1-4p^n+2p^{2n}
    +(Mp^2)^n\left(-1+ p^2\frac{M-1}{M-p^2}\left[1-\left(\frac {p^2}M\right)^n\right]\right).\qquad \qedhere
 \end{align*}
\end{proof}

It is now easy to derive explicit expressions for the limit functionals $\osV^c_k(K)=\lim_{n\to\infty} r^{(D-k)n}\E V_k(D_{n})$ of fractal percolation $K$ in $\R$ {for all possible parameters}.
 \begin{cor} \label{cor:5.2}
    {Let $K$ be a fractal percolation on the interval $I=[0,1]$ with parameters $M\in\N_{\geq 2}$ and $p\in(0,1]$. If $K_n$ is the $n$-th construction step and $D_{n}:=\overline{I\setminus K_{n}}$, then
    \begin{align*}
      \lim_{n\to\infty} \E V_1(D_{n})=\begin{cases}1, \text{ for }  p\in(0,1),& (\text{while } \osV^c_1(K)=%\displaystyle\lim_{n\to\infty} r^{(D-1)n}\E V_1(D_{n})=
      \infty),\\ 0, \text{ for } p=1, &(\text{which equals } \osV^c_1(K) \text{ in this case}). \end{cases}
        \end{align*}
        Moreover, %for $p\in(1/M,1]$,
  \begin{align*}
    \osV^c_0(K)=\begin{cases}\frac{M(1-p)}{M-p}, & \text{ for } p\in(1/M,1], \quad (\text{which equals } \osV_0(K), \text{ cf.\ Corollary~\ref{cor:Vk-dim1}}),\\
    \frac{2M+1}{M+1}, & \text{ for } p=1/M,\\ \infty,& \text{ for } p\in(0,1/M).\end{cases}
    \end{align*}
    }
  \end{cor}
  \begin{proof}
    %Recall from \eqref{eq:dimF}, that %the Hausdorff dimension of $K$ (provided $K\neq \emptyset$) is given by
   { Since here $D=\frac{\log Mp}{\log M}$ (cf.\ \eqref{eq:dimF}), which implies $Mp=r^{-D}$, we infer from Theorem~\ref{thm:MPcapMP-compl} that, for any $n\in\N$, $r^{(D-1)n}\E V_1(D_{n})=p^{-n}-1$ and
    \begin{align*}
     r^{Dn}\E V_0(D_{n})= 1-p\frac{M-1}{M-p}\left[1-\left(\frac {p}M\right)^n\right]+(Mp)^{-n}(1-2p^n).
  \end{align*}
  Letting now $n\to\infty$, the stated limits follow at once.}
  \end{proof}

\paragraph*{The case $d=2$.} %For Mandelbrot percolation
In $\R^2$, formula \eqref{eq:Zk-limit-general2} in Theorem~\ref{thm:Zk-limit-general2} reduces to
  \begin{align}\label{eq:Yk-limit-dim2-1}
    \osV^c_k(F) %\lim_{n\to\infty} r^{n(D-k)} &\E V_k(C_{n})
    =q_{2,k} \frac{M^{2-k}(1-p)}{M^{2-k}p-1}-E_1-E_2+E_3-E_4,
  \end{align}
  where
  \begin{align*}
     E_1&:=2M(M-1)\sum_{n=1}^\infty  r^{n(D-k)} \E V_k(C_n^1\cap C_n^4),\notag\\
    E_2&:=2(M-1)^2 \sum_{n=1}^\infty r^{n(D-k)} \E V_k(C_n^1\cap C_n^2), \notag \\
     E_3&:=4(M-1)^2 \sum_{n=1}^\infty r^{n(D-k)} \E V_k(C_n^1\cap C_n^2\cap C_n^3), \notag \\
      E_4&:=(M-1)^2 \sum_{n=1}^\infty r^{n(D-k)} \E V_k\left(\bigcap_{j=1}^4 C_{n}^j\right)%\cap F_n^2\cap F_n^3\cap F_n^4)
      .\notag
  \end{align*}
  Here the sets $C_n^1, \ldots, C_n^4$ are four of the $M^2$ sets $C_{n}^j=\overline{J_j\setminus F_n^j}$ chosen such that the corresponding sets $J^j$, $j=1,\ldots,4$ intersect in a point $x$ and are numbered as indicated in Figure~\ref{fig:J_j}. The factor in front of the summation in each $E_i$ indicates how many times this particular intersection configuration occurs in the union $C_{n}=\bigcup_{j=1}^{M^2} C_{n}^j$ taking into account all symmetries.
 %\begin{figure}[t]
%      \begin{center}
%        \includegraphics[width=2cm]{MP-Cn-subdiv-2.pdf}
%      \end{center}
%       \caption{\label{fig:Cn_changed order} The sets $C_{n}^j$ around a potential intersection point $x$ numbered in a particular convenient way as needed for the proof of Lemma~\ref{lem:E2-E4}.}
%    \end{figure}

{Theorem~\ref{thm:Zk-limit-general2} asserts that, for $k=0$, formula \eqref{eq:Yk-limit-dim2-1} is valid for all $p\in(1/M^2,1]$, and for $k=1$ for all $p\in(1/M,1]$.} %, provided that all the sums in the terms $E_1,\ldots, E_4$ are finite (which we show below).
    While for $E_1$ we will again employ the one-dimensional case, the last three summands $E_2, E_3$ and $E_4$ vanish for $k=1$, since the involved intersections contain at most one point. For $k=0$ these terms can be obtained by direct inspection of the intersections of the  $C_{n}^j$.

    \begin{lem} \label{lem:E2-E4}
    Suppose $p>1/M^2$. Then, for $k=0$, the terms $E_2, E_3$ and $E_4$ in \eqref{eq:Yk-limit-dim2-1} are given by
    \begin{align*}
      E_2&= 2(M-1)^2\left(\frac{1}{M^2p-1}-\frac 2{M^2-1}+\frac {p}{M^2-p}\right),\\
      E_3&=4(M-1)^2\left(\frac{1}{M^2p-1}-\frac 3{M^2-1}+\frac {3p}{M^2-p}-\frac{p^2}{M^2-p^2}\right), \text{ and}\\
      E_4&= (M-1)^2\left(\frac{1}{M^2p-1}-\frac 4{M^2-1}+\frac {6p}{M^2-p}-\frac{4p^2}{M^2-p^2}+\frac{p^3}{M^2-p^3}\right).
    \end{align*}
    \end{lem}
    \begin{proof}
     In all three intersection configurations of the sets $C_{n}^j$ considered here, the intersection contains at most one point, $x$, cf.\ Figure~\ref{fig:J_j}.
     % For this computation it is convenient to encode the possible one-point intersections of 2 (diagonal case), 3 or 4 of the sets $C_{n}^j$ lying in neighbouring boxes of a potential intersection point $x$ in the way indicated in Figure~\ref{fig:J_j}.
    For each of the sets $C_{n}^j$ to contain  $x$ it is necessary, that at least in one of the sets $F^j_k$, $k=1,\ldots, n$ the $k$-th level subsquare intersecting $x$ is discarded (which happens with probability $1-p^n$). Thus, by independence, we obtain for the intersections of $\ell=2,3$ or $4$ of these  sets
     \begin{align*}
       \E V_0\Big(\bigcap_{j=1}^\ell C_{n}^j\Big)=\P\Big(\bigcap_{j=1}^\ell C_{n}^j=\{x\}\Big)=(1-p^n)^\ell,
     \end{align*}
     and therefore
     \begin{align} \label{eq:lem:E2-E4}
       \sum_{n=1}^\infty r^{nD} \E V_0\Big(\bigcap_{j=1}^\ell C_{n}^j\Big)&=\sum_{k=0}^\ell \binom{l}{k} (-1)^k\frac{p^{k-1}}{M^2-p^{k-1}}.
     \end{align}

     Indeed, employing the binomial theorem and the relation $r^{-D}=M^2p$, we get
     \begin{align*}
       \sum_{n=1}^\infty r^{nD} \E V_0\Big(\bigcap_{j=1}^\ell C_{n}^j\Big)&=\sum_{n=1}^\infty (M^2p)^{-n}(1-p^n)^\ell=\sum_{k=0}^\ell \binom{l}{k}(-1)^k \sum_{n=1}^\infty \left(\frac{p^{k-1}}{M^2}\right)^{n},
       %&=\sum_{k=0}^\ell \binom{l}{k} (-1)^k\frac{p^{k-1}}{M^2-p^{k-1}}.
     \end{align*}
     where in the last expression all the geometric series converge due to the assumption $p>1/M^2$. Computing these series yields the expression stated in \eqref{eq:lem:E2-E4}. Finally, the assertion of the lemma follows by plugging \eqref{eq:lem:E2-E4} into the expressions for $E_2$, $E_3$ and $E_4$ given by \eqref{eq:Yk-limit-dim2-1}.
    \end{proof}

   The expressions derived in Theorem~\ref{thm:MPcapMP-compl} for intersections of fractal percolations in one dimension will now be used to compute the expected intrinsic volumes of the (1-dimensional) intersections $C_n^1\cap C_n^4$ appearing in the term $E_1$ in formula \eqref{eq:Yk-limit-dim2-1} for  fractal percolation in $\R^2$.
  \begin{prop} \label{prop:CncapCn}
    Let $F$ be a fractal percolation in $\R^2$ with parameters $M\in\N_{\geq 2}$ and $p\in(0,1]$. %Let $F_n$ and $C_{n}=\cl{J\setminus F_n}$, $n\in\N_0$, be as defined in.
Then, for any $n\in\N$,
  \begin{align*}
       \E V_1(C_n^1\cap C_n^4)&=\frac 1M\left(1-2p^{n}+p^{2n}\right) \qquad \text{ and }\\
       \E V_0(C_n^1\cap C_n^4)&=2(Mp)^n\left(\frac{1-p}{M-p}+\frac{M-1}{M-p}\left(\frac {p}M\right)^{n}\right)+1-4p^{n}+2p^{2n}\\&\quad-(Mp^2)^n\left(\frac{1-p^2}{M-p^2}+\frac{M-1}{M-p^2}\left(\frac {p^2}M\right)^{n}\right).%=2p\E V_0(K_{n-1}^{(1)})-p^2\E V_0(K_{n-1}^{(1)}\cap K_{n-1}^{(2)})\\ &\qquad +1-4p^n+2p^{2n}.
  \end{align*}
\end{prop}
\begin{proof}
  Let $K^{(1)}, K^{(2)}$ be two independent fractal percolations on the interval $I=[0,1]$ with the same parameters $M$ and $p$ as $F$ and independent of $F$. For $n\in\N_0$, let $K_n^{(i)}$ be the $n$-th step of the construction of  $K^{(i)}$, $i=1,2$ and let $D_n^{(i)}:=\cl{I\setminus K_n^{(i)}}$ (just as in Corollary~\ref{cor:MPcapMP-compl}). Denote by $\widetilde K_n^{(i)}$, $i=1,2$, the random set which equals $K_n^{(i)}$ with probability $p$ and is empty otherwise.
  Recalling from \eqref{eq:MpcapMP-dim1-2} that in distribution
  $$
  F_n^1\cap F_n^4=\psi (\widetilde K_{n-1}^{(1)}\cap \widetilde K_{n-1}^{(2)}),
  $$
  we infer that also the following equation holds in distribution, since $C_{n}^j$ is determined by $F_n^j$ and similarly $D_{n-1}^{(i)}$ is determined by $K_{n-1}^{(i)}$. For each $n\in\N$, we have
  \begin{align}
     \label{eq:CncapCn} C_n^1\cap C_n^4=\psi (\hat D_{n-1}^{(1)}\cap \hat D_{n-1}^{(2)}),
  \end{align}
  where $\hat D^{(i)}_n$ is the random set which equals $D^{(i)}_n$ with probability $p$ and $I$ with probability $1-p$. This implies for each $n\in\N_0$,
  \begin{align*}
    \E V_0(C^1_{n+1}\cap C^4_{n+1})&=\E V_0(\hat D_n^{(1)}\cap \hat D_n^{(2)})\\
    &=p^2 \E V_0(D_n^{(1)}\cap D_n^{(2)})+ p(1-p) \E V_0(D_n^{(1)}\cap I)\\
    &\qquad +p(1-p) \E V_0(I\cap D_n^{(2)})+(1-p)^2 \E V_0(I\cap I)\\
    &=p^2 \E V_0(D_n^{(1)}\cap D_n^{(2)})+2 p(1-p) \E V_0(D_n^{(1)})+(1-p)^2.
  \end{align*}
Employing now the formulas derived in Theorem~\ref{thm:MPcapMP-compl} for $\E V_0(D_n^{(1)}\cap D_n^{(2)})$ and $\E V_0(D_n^{(1)})$,
we obtain for each $n\in\N_0$
\begin{align*}
   \E &V_0\left(C^1_{n+1}\cap C^4_{n+1}\right)\\
   &=p^2 \left(%2\E V_0(K_n^{(1)})-\E V_0(K_n^{(1)}\cap K_n^{(2)})+1-4p^n+2p^{2n}
  2 (Mp)^n\left(1-p\frac{M-1}{M-p}\left[1-\left(\frac pM\right)^n\right]\right)+1-4p^n+2p^{2n}\right.\\
    &\qquad \left. +(Mp^2)^n\left(-1+p^2\frac{M-1}{M-p^2}\left[1-\left(\frac {p^2}M\right)^n\right]\right)
   \right)\\
   &\qquad +2p(1-p)\left((Mp)^n\left(1-p\frac{M-1}{M-p}\left[1-\left(\frac {p}M\right)^n\right]\right)+1-2p^n \right) +(1-p)^2\\
   &=2p (Mp)^n\left(1-p\frac{M-1}{M-p}\left[1-\left(\frac {p}M\right)^n\right]\right)+1-4p^{n+1}+2p^{2(n+1)}\\
   &\qquad +p^2 (Mp^2)^n\left(-1+p^2\frac{M-1}{M-p^2}\left[1-\left(\frac {p^2}M\right)^n\right]\right),
\end{align*}
 where we combined some of the terms to get to the last expression. Replacing now $n+1$ by $n$, this simplifies to
 \begin{align*}
   \E V_0\left(C_n^1\cap C_n^4\right)&=\frac 2M (Mp)^n\left(1-p\frac{M-1}{M-p}\left[1-\left(\frac {p}M\right)^{n-1}\right]\right)+1-4p^{n}+2p^{2n}\\
   &\qquad +\frac 1M (Mp^2)^n\left(-1+p^2\frac{M-1}{M-p^2}\left[1-\left(\frac {p^2}M\right)^{n-1}\right]\right)\\
   &=2(Mp)^n\left(\frac{1-p}{M-p}+\frac{M-1}{M-p}\left(\frac {p}M\right)^{n}\right)+1-4p^{n}+2p^{2n}\\
   &\qquad-(Mp^2)^n\left(\frac{1-p^2}{M-p^2}+\frac{M-1}{M-p^2}\left(\frac {p^2}M\right)^{n}\right),
\end{align*}
 for any $n\in\N$, completing the proof of the formula for $V_0$ in Proposition~\ref{prop:CncapCn}. The formula for $V_1$ follows similarly from \eqref{eq:CncapCn} using the corresponding formulas from Theorem~\ref{thm:MPcapMP-compl}:
\begin{align*}
    \E V_1&(C^1_{n+1}\cap C^4_{n+1})=\E V_1(\psi(\hat D_n^{(1)}\cap \hat D_n^{(2)}))=\frac 1M \E V_1(\hat D_n^{(1)}\cap \hat D_n^{(2)})\\
        &=\frac 1M\left( p^2 \E V_1(D_n^{(1)}\cap D_n^{(2)})+2 p(1-p) \E V_1(D_n^{(1)})+(1-p)^2\right)\\
        &=\frac 1M\left(p^2\left[1-2p^n+p^{2n}\right]+2 p(1-p) \left[1-p^n\right]+(1-p)^2\right)\\
        &=\frac 1M\left(1-2p^{n+1}+p^{2(n+1)}\right). \qedhere %p^2\E V_1(K_n^{(1)}\cap K_n^{(2)})-2p\E V_1(K_n^{(1)})+1.
  \end{align*}
\end{proof}

    \begin{cor} \label{cor:E1}
    If $p>1/M^2$, then for $k=0$ the term $E_1$ in \eqref{eq:Yk-limit-dim2-1} is given by
    \begin{align*}
      E_1&=\frac{2M(M-1)}{M-p}\left(\frac{2(1-p)}{M-1}+\frac{2(M-1)p}{M^2-p}-\frac{p(1-p^2)}{M-p^2}\right)\\
   &\quad-\frac{2M(M-1)^2 p^3}{(M-p^2)(M^2-p^3)}
   +\frac{2M(M-1)}{M^2p-1}-\frac {8M}{M+1}+\frac{4M(M-1)p}{M^2-p}.
    \end{align*}
    Similarly, if $p>1/M$, then for $k=1$, $E_1=2(M-1)\left(\frac 1{Mp-1}-\frac 2{M-1}+\frac p{M-p}\right)$.
    \end{cor}
    \begin{proof}
      To determine $E_1$ for $k=0$, we multiply the expression derived in Proposition~\ref{prop:CncapCn} for $\E V_0\left(C_n^1\cap C_n^4\right)$ by $r^{Dn}=(M^2p)^{-n}$ and sum over $n$ to obtain
\begin{align} \label{eq:cor:E1}
   E_1&=2M(M-1)\sum_{n=1}^\infty (M^2p)^{-n}\left[2(Mp)^n\left(\frac{1-p}{M-p}+\frac{M-1}{M-p}\left(\frac {p}M\right)^{n}\right)\right.\\
   &\notag \qquad\quad\left.-(Mp^2)^n\left(\frac{1-p^2}{M-p^2}+\frac{M-1}{M-p^2}\left(\frac {p^2}M\right)^{n}\right)+1-4p^{n}+2p^{2n}\right],
\end{align}
and the expression stated above is then derived by computing the various geometric series (which do all converge due to the assumption $p>1/M^2$, justifying thus in particular the above interchange of summations) and combining some of the terms. For $k=1$, the stated expression for $E_1$ follows similarly by multiplying the expression for $\E V_1\left(C_n^1\cap C_n^4\right)$ from Proposition~\ref{prop:CncapCn} by $r^{(D-1)n}=(Mp)^{-n}$ and summing over $n$. The involved series converge due to the assumption $p>1/M$.
    \end{proof}

Now we are ready to prove Theorem~\ref{thm:Vck-limit-dim2}. For convenience, we repeat the statement here, and we add a corresponding formula for the limit $\osV^c_1(F)$ of the rescaled boundary lengths:

 \begin{prop} \label{prop:Vck-limit-dim2} %{thm:Yk-limit-dim2}
Let $F$ be a fractal percolation in $\R^2$ with parameters $M\in\N_{\geq 2}$ and $p\in [0,1]$.
  Then, for any $p>1/M^2$,
  \begin{align*}
       \osV^c_0(F)&= M^2(1-p)\frac{p^3+(M-1)p^2+(M-1)p-M}{(M^2-p^3)(M-p)}.
  \end{align*}
  %\begin{align*}
%       \osV^c_0(F)&= M^2(1-p)\frac{Mp^2+p^3+Mp-p^2-M-p}{(M^2-p^3)(M-p)}.
%  \end{align*}
  Moreover, for any $p>1/M$,
  \begin{align*}
       \osV^c_1(F)&= 2M \frac{1-p}{M-p}\quad\left(=\osV_1(F), \quad \text{ cf.~\eqref{thm:Vk-limit-dim2}}\right).
  \end{align*}
\end{prop}
\begin{proof}[Proof of Proposition~\ref{prop:Vck-limit-dim2} and thus in particular of Theorem~\ref{thm:Vck-limit-dim2}]
All one has to do is to insert the expressions for $E_1,\ldots, E_4$ obtained in Lemma~\ref{lem:E2-E4} and Corollary~\ref{cor:E1} into formula \eqref{eq:Yk-limit-dim2-1} for $\osV^c_k(F)$.
For $k=0$, we obtain (recalling that $q_{2,0}=V_0(J)=1$)
\begin{align} \label{eq:Vck-limit-dim2-proof}
    \osV^c_0(F)&=\frac{M^{2}(1-p)}{M^{2}p-1}-E_1-E_2+E_3-E_4\\
   \notag &=\frac{M^{2}(1-p)}{M^{2}p-1}-\frac{2M(M-1)}{M-p}\left(\frac{2(1-p)}{M-1}+\frac{2(M-1)p}{M^2-p}-\frac{p(1-p^2)}{M-p^2}\right)\\
  \notag  &\quad+\frac{2M(M-1)^2 p^3}{(M-p^2)(M^2-p^3)}
   -\frac{2M(M-1)}{M^2p-1}+\frac {8M}{M+1}-\frac{4M(M-1)p}{M^2-p}\\
   \notag &\quad -E_2+E_3-E_4,
  \end{align}
where
\begin{align*}
   -E_2+E_3-E_4&=(M-1)^2\left(\frac{1}{M^2p-1}-\frac 4{M^2-1}+\frac {4p}{M^2-p}-\frac{p^3}{M^2-p^3}\right).
\end{align*}
Fortunately, this can be simplified to the expression stated above.
Similarly, we get for $k=1$ (taking into account that $E_2=E_3=E_4=0$ in this case)
  \begin{align*}
       \osV^c_1(F)&= %2M\frac{1-p}{Mp-1}-E_1
       2M\frac{1-p}{Mp-1}-2(M-1)\left(\frac{1}{Mp-1}-\frac{2}{M-1}+\frac{p}{M-p}\right),
  \end{align*}
  which simplifies to the expression stated {above}. %for $\osV^c_1(F)$.
  This completes the proof. % of Proposition~\ref{prop:Vck-limit-dim2}.
\end{proof}
\begin{rem} {(On the speed of convergence)} \label{rem:speed2}
  From the proof of Proposition~\ref{prop:Vck-limit-dim2}, we also get explicit expressions for the expected intrinsic volumes of the approximation sets $C_m$ for each $m\in\N$. To determine $\overline{v}^c_k(m):=r^{m(D-k)}\E V_k(C_m)$, $m\in\N$, it is enough to truncate all the sums in formula \eqref{eq:Yk-limit-dim2-1} after the $m$-th term (including the very first one which appears already in summed form in \eqref{eq:Yk-limit-dim2-1}, cf.~\eqref{eq:Yk-proof}).
  %and compute the resulting finite geometric series.
  For $k=0$, we obtain for any $m\in\N$,
  \begin{align*}
  \overline{v}^c_0(m)&=\frac{1-p}p\sum_{n=0}^m r^{nD}-E_1(m)-2E_2(m)+4 E_3(m)-E_4(m),
\end{align*}
where $E_1(m)$, the truncated term corresponding to $E_1$, can be read off from equation~\eqref{eq:cor:E1} in the proof of Corollary~\ref{cor:E1}:
 \begin{align*}
     E_1(m)&:=%2M(M-1)\sum_{n=1}^m r^{n(D-k)} \E V_k(C_n^1\cap C_n^4),\notag\\
     2M(M-1)\sum_{n=1}^m \left[2\frac{1-p}{M-p}\left(\frac 1M\right)^n+2\frac{M-1}{M-p}\left(\frac {p}{M^2}\right)^{n}-\frac{1-p^2}{M-p^2}\left(\frac pM\right)^n\right.\\
   &\qquad\qquad\qquad\left.-\frac{M-1}{M-p^2}\left(\frac {p^3}{M^2}\right)^{n}+\left(\frac 1{M^2p}\right)^n - 4 \left(\frac 1{M^2}\right)^n +2\left(\frac p{M^2}\right)^n\right].%\\
   %&= 2M(M-1)\left[...+2\frac{2M-1-p}{M-p}\frac p{M^2-p}\left[1-\left(\frac{p}{M^2}\right)^m\right]-\frac{1-p^2}{M-p^2}\frac{p}{M-p}\left[1-\left(\frac pM\right)^m\right]\right.\\
   \end{align*}
Similarly, $E_\ell(m)$, $\ell=2,3,4$, are derived by truncating the corresponding sums $E_\ell$ and computing the resulting finite geometric sums, cf.~\eqref{eq:lem:E2-E4}:
  \begin{align*}
          E_\ell(m)&:=(M-1)^2 \sum_{n=1}^m (M^2p)^{-n}(1-p^n)^\ell\\&=(M-1)^2\sum_{k=0}^\ell \binom{\ell}{k} (-1)^k\frac{p^{k-1}}{M^2-p^{k-1}}\left[1-\left(\frac{p^{k-1}}{M^2}\right)^m\right].
   %  E_3(m)&:=4(M-1)^2 \sum_{n=1}^m r^{n(D-k)} \E V_k(C_n^1\cap C_n^2\cap C_n^3), \notag \\
%      E_4(m)&:=(M-1)^2 \sum_{n=1}^m r^{n(D-k)} \E V_k\left(\bigcap_{j=1}^4 C_{n}^j\right)%\cap F_n^2\cap F_n^3\cap F_n^4)
     \end{align*}
Computing all the finite geometric sums, we get for each of the terms in equation \eqref{eq:Vck-limit-dim2-proof} a corresponding one for $\overline{v}^c_0(m)$ with a factor of the form $\left(1-q^m\right)$ for a suitable $q$ (just as in the last line of the formula above). The constant terms add up to $\osV^c_0(F)$, such that we end up with the following exact expansion in $m$:
\begin{align} \label{eq:soc}
  \overline{v}^c_0(m)&=\osV^c_0(F)+\frac{4M(1-p)}{M-p} M^{-m} - \frac{2M(M-1)p(1-p^2)}{(M-p)(M-p^2)} M^{(D-3)m} + M^{-Dm}\notag\\
  &\qquad -4 M^{-2m}+\frac{4p(M-1)}{M-p} M^{(D-4)m}+ \tilde c M^{(3D-8)m},
\end{align}
where $\tilde c:=\frac{(M-1)p^3}{M^2-p^3}\frac{(M-1)(M-p^2)-2M(M-p)}{M-p^2}$.
It is easy to see that this sequence converges again exponentially fast to $\osV^c_0(F)$ as $m\to\infty$.
\end{rem}
\begin{rem} {(On fractal subdimensions)} \label{rem:subdim}
Multiplying \eqref{eq:soc} by $M^{Dm}$, we obtain an exact expansion for $\E V_0(C_m)$:
\begin{align*}
   \E V_0(C_m)&=\osV^c_0(F) M^{Dm}+\frac{4M(1-p)}{M-p} M^{(D-1)m} - \frac{2M(M-1)p(1-p^2)}{(M-p)(M-p^2)} M^{(2D-3)m}\\
  &\qquad + 1-4 M^{(D-2)m}+\frac{4p(M-1)}{M-p} M^{(2D-4)m}+ \tilde c M^{(4D-8)m}.
\end{align*}
Since $D<2$ for $p<1$, the last three terms vanish as $m\to\infty$. The remaining terms determine the \emph{subdimensions} of $F$ in the sense of \cite{MeckeSchoenhoefer15}. We obtain,
\begin{align*}
   \E V_0(C_m)
  &=\osV^c_0(F) M^{Dm}+c_2 M^{D_2 m} +c_3 M^{D_3 m}+1+o(1),
\end{align*}
as $m\to\infty$, where $c_2:=\frac{4M(1-p)}{M-p}$ is the amplitude of the first subdimension $D_2:=D-1$ and $c_3:=- \frac{2M(M-1)p(1-p^2)}{(M-p)(M-p^2)}$ is the amplitude of second subdimension $D_3:=2D-3=\frac{\log Mp^2}{\log M}$. Recall from Proposition~\ref{prop:dim_K1capK2} and Remark~\ref{rem:K1capK2}, that $D_3$ (which is positive for $p>1/\sqrt{M}$) is the dimension of the intersection of {two} copies of $F$ constructed in neighboring squares sharing a common side. Similarly $D_2$ (which is positive for $p>1/M$) is the dimension of a fractal percolation on an interval with the same parameters as $F$ or equally the dimension of $F\cap \partial[0,1]^2$. Hence two subdimensions appear for these random fractals as suggested by \cite{Knuefing05,MeckeSchoenhoefer15} and they carry geometric meaning as in the deterministic setting studied there. %(this set is only exists for $p>1/\sqrt{M}$).
\end{rem}

\section{Proof of Propositions~\ref{prop:Rk-conv} and \ref{prop:Rk-conv2}} \label{sec:appendix}

Let $n\in\N$ and $W_1, W_2,\ldots$ be unions of subcubes of $J=[0,1]^d$ of level $n$. More precisely, if $\Omega^{1},\Omega^{2},\ldots$ are arbitrary subsets of $\{1,\ldots,M^d\}^n$, then we let
\begin{align} \label{eq:def-Wi}
   W_i:=\bigcup_{\sigma\in \Omega^i} J_\sigma,\qquad i\in\N.
\end{align}
Our first aim is to establish a general bound on the curvature of the intersection %$W_1\cap W_2$ and, more generally, of the intersection
$W_1\cap W_2\cap\ldots \cap W_\ell$ for an arbitrary number $\ell\in\N_{\ge 2}$ of these sets. For this we will employ an estimate from \cite{W08}. Recall from \cite{W08} that for any finite family $\mathcal{X}=\{X_1,\ldots, X_m\}$ of sets, the \emph{intersection number} $\Gamma=\Gamma(\mathcal{X})$ is defined by
\begin{align*}
  \Gamma:=\max_{i\in\{1,\ldots,m\}}|\{j: X_j\cap X_i\neq\emptyset\}|.
\end{align*}
If $\Gamma$ is small compared to $m$, then the following estimate is particularly useful, which is a special case of \cite[Corollary 3.0.5]{W08} for a family of convex sets.
%%%%%%%%%%%%%%%%%%%%%%%%%%%%%%%%%%
\begin{lem} %\cite[Corollary 3.0.5]{W08}
\label{lem:ringcor}
Let $\{X_1,\ldots, X_m\}$ be a family of compact, convex subsets of $\R^d$ and let %, $X=\bigcup_{j=1}^m X_j$,
$\Gamma$ be its intersection number. % of the family $\{X_1,\ldots, X_m\}$.
% such that for all $T\subset\{1,\ldots,m\}$
%\[C^\var_k(K_j)\bigcap_{j\in T}K_j)\le b.\]
Then, for any $k\in\{0,\ldots,d\}$,
\[C^\var_k(\bigcup_{j=1}^m X_j)\le m 2^\Gamma b_k,\]
where $b_k:=\max\{C_k(X_j):j=1,\ldots,m\}$.
\end{lem} %\todo{Add Cor~3.0.5 from \cite{W08}?}
\begin{proof}
  Since the $X_i$ are convex, any intersection $X_I:=\bigcap_{i\in I} X_i$, $I\subseteq\{1,\ldots,m\}$ is also convex and contained in any of the sets $X_i$, $i\in I$. Therefore, the monotonicity and positivity of the intrinsic volumes implies that $C_k^{\var}(X_I)=C_k(X_I)\leq \max_{i\in I} C_k(X_i)\leq b_k$. Thus (for $B:=\R^d$) the assumptions of \cite[Corollary 3.0.5]{W08} are satisfied with $b:=b_k$ and the assertion follows.
\end{proof}
Recall that $r=1/M$.
\begin{lem} \label{lem:intersect-num}
   There is a constant $c_{d,k}$ %(independent of $n$ and $\ell$)
   such that for any $n\in\N$ and $\ell\in\N_{\geq 2}$ and any collection $W_1,\ldots,W_\ell$ of unions of cubes of level $n$ as defined in \eqref{eq:def-Wi} the following estimate holds
   \begin{align}\label{eq:lem-intersect-num1}
     C_k^{\var}(W_1\cap \ldots \cap W_\ell)\leq c_{d,k} |\Omega^{1}| r^{kn},
   \end{align}
   where $|\Omega^{1}|$ is the number of cubes in $W_1$. %is less than $N$, then $C_k^{\var}(W_1\cap \ldots\cap W_l)\leq c_{d,k} N r^kn$.
\end{lem}
Note that $|\Omega^{1}|\leq M^{dn}=r^{-dn}$, which implies that the right hand side of \eqref{eq:lem-intersect-num1} is always bounded from above by $c_{d,k} r^{(k-d)n}$.
\begin{proof}
   First let $\ell=2$. We write the intersection $W_1\cap W_2$ as a union of convex sets. It is clear that each set $J_\sigma$ in the union $W_1$ intersects at most $3^d$ cubes (the neighboring ones) from the union $W_2$. Let ${\Omega}^{2}_{\sigma}\subset \Omega^{2}$ be the set of indices of the cubes from $W_2$ intersecting $J_\sigma$. Then $|\Omega^{2}_\sigma|\leq 3^d$ and we have
   \begin{align*}
     W_1\cap W_2=\bigcup _{\sigma\in \Omega^{1}} \left(J_\sigma\cap \bigcup _{\omega\in \Omega^{2}_\sigma} J_\omega\right)
     =\bigcup _{\sigma\in \Omega^{1}} \bigcup _{\omega\in \Omega^{2}_\sigma} \left(J_\sigma\cap  J_\omega\right).
   \end{align*}
   This way we have represented $W_1\cap W_2$ as a union of at most $|\Omega^{1}|\cdot 3^d\leq(M^d)^n 3^d$ convex sets $R_{\sigma,\omega}:=J_{\sigma}\cap J_{\omega}$.
   Note that each of the sets $R_{\sigma,\omega}$ is the intersection of two cubes and thus a $k$-face of some cube of level $n$ (of some dimension $k\in\{0,\ldots,d\}$). We may reduce the number of sets in this representation by deleting the double occurrences of any face without changing the union set. Then the reduced family $\mathcal{F}\subset\{R_{\sigma,\omega}\}$ has an intersection number $\Gamma=\Gamma(\mathcal{F})$ bounded from above by $3^d$ times the number of faces of a cube in $\R^d$ (which also equals $3^d$). Indeed, each set $R\in\mathcal{F}$ is contained in a cube of dimension $d$ and any other set $R'\in\mathcal{F}$ intersecting $R$ must be a face of the same cube or of one of the neighboring cubes. Note also that any of the sets $R\in\mathcal{F}$ is convex and contained in a cube of sidelength $r^n$. Therefore,
   \begin{align*}
     C_k^{\var}(R)= C_k(R)\leq  C_k(r^n J) = r^{kn} C_k(J)=r^{kn} q_{d,k},
   \end{align*}
   where he have used the monotonicity, motion invariance and homogeneity of the intrinsic volumes.  Now we can apply Lemma~\ref{lem:ringcor} to the family $\mathcal{F}$ consisting of $m\leq(M^d)^n 3^d$ sets and satisfying $b_k:=\max\{C_k(R):R\in \mathcal{F}\}\leq r^{kn} q_{d,k}$. We obtain
   \begin{align*}
      C_k^{\var}(W_1\cap W_2)\leq |\Omega^{1}| 3^d 2^{3^{2d}} r^{kn} q_{d,k}= c_{d,k} |\Omega^{1}|r^{kn},
   \end{align*}
   where the constant $c_{d,k}:= 3^d 2^{3^{2d}} q_{d,k}$ is independent of $n$. This proves the case $\ell=2$.
   For the general case, fix some $\ell>2$ and note that $W_1\cap\ldots\cap W_\ell$ can be represented by
\begin{align} \label{eq:lem-intersect-num}
     W_1\cap \ldots \cap W_\ell
     =\bigcup _{\sigma\in \Omega^{1}} \bigcup _{\omega_2\in \Omega^{2}_\sigma} \ldots \bigcup _{\omega_\ell\in \Omega^{\ell}_\sigma} J_\sigma\cap   J_{\omega_2}\cap \ldots\cap J_{\omega_\ell},
   \end{align}
   where similarly as before $\Omega^{j}_\sigma$ is the family of those words $\omega\in\Omega^{j}$ for which $J_\sigma \cap J_{\omega} \neq\emptyset$, $j=2,\ldots,\ell$. Now observe that $J_\sigma\cap J_{\omega_2}\cap \ldots\cap J_{\omega_\ell}$ is a finite intersection of cubes of the grid and thus a $k$-face of $J_\sigma$ (of some dimension $k\in\{0,\ldots,d\}$) -- if not empty. Hence, for fixed $\sigma$, there are at most $3^d$ distinct sets in the union corresponding to the faces of $J_\sigma$. Deleting all multiplicities such that no set appears more than once in the union on the right of \eqref{eq:lem-intersect-num}, we again end up with a representation of $W_1\cap \ldots \cap W_\ell$ by at most $|\Omega^{1}|\cdot 3^d$ convex sets and, as before, the intersection number of the reduced family will not exceed $3^{2d}$. Hence \eqref{eq:lem-intersect-num1} follows again from Lemma~\ref{lem:ringcor} with the same constant $c_{d,k}$ as before. This completes the proof for arbitrary integers $\ell\geq 2$.
   %We can now again reduce the number of sets in this representation such that no face appears twice without changing the union set $W_1\cap \ldots \cap W_l$.
\end{proof}

\begin{rem} \label{rem:bd-estimate} {If the union set $W_1$ contains many `interior' cubes of the intersection $\bigcap_{j=1}^\ell W_j$, then the above estimate can be improved. Recall that, for  $k\leq d-1$, the $k$-th curvature measure of any set is concentrated on its boundary. Therefore, in order to bound the total curvature variation of the set $\bigcap_{j=1}^\ell W_j$, it is enough to represent this set near its boundary by a union of cubes. To this end, let $\Omega_\partial^1\subseteq \Omega^1$ be the set of those cubes in $\Omega^1$ which have a nonempty intersection with the boundary $\partial\bigcap_{j=1}^\ell W_j$.
  %it is enough in the proof of Lemma~\ref{lem:intersect-num} to represent
  %the boundary of $W_1\cap\ldots \cap W_\ell$ as a union of cubes in the following way:
  Let $A$ be the set defined by the right hand side of \eqref{eq:lem-intersect-num} when we replace the set $\Omega^1$ in the first union  by $\Omega^1_\partial$. Then $A$ is obviously a union of cubes and a subset of $\bigcap_{j=1}^\ell W_j$. Moreover, it has the property that $\partial\bigcap_{j=1}^\ell W_j\subset \partial A$ and that any point $x\in\partial A$ that is not in $\partial\bigcap_{j=1}^\ell W_j$ has positive distance (at least $r^n$) to $\partial\bigcap_{j=1}^\ell W_j$. Therefore, since curvature measures are locally determined, we conclude that (for any $k\in\{0,\ldots, d-1\}$)
  \begin{align*}
     C_k^\var(W_1\cap\ldots \cap W_\ell)\leq C_k^\var(A)\leq c_{d,k} |\Omega^{1}_\partial| r^{kn},
  \end{align*}
  where the second inequality follows by applying the same argument to $A$ that we applied in the proof of Lemma~\ref{lem:intersect-num} to the set on the right hand side of \eqref{eq:lem-intersect-num}.
  }
\end{rem}

{With Lemma~\ref{lem:intersect-num} and Remark~\ref{rem:bd-estimate}
at hand, we are now in a position to prove Propositions~\ref{prop:Rk-conv} and \ref{prop:Rk-conv2}.}
\begin{proof}[Proof of Propositions~\ref{prop:Rk-conv} and \ref{prop:Rk-conv2}]
  {First note that in both statements the second assertion is an immediate consequence of the first one, which is due to the fact that $|V_k(K)|\leq C_k^\var(K)$ for any polyconvex set $K$. In order to prove the first assertions in  the two propositions, we fix a set $T\subseteq\{1,\ldots,M^d\}$, $|T|\geq 2$ and let $U:=\bigcap_{j\in T} J_j$ (where $J_j$ is the cube of sidelength $r=1/M$ containing $F^j$). $U$ is a cube of some dimension $u\in\{0,\ldots, d-1\}$ and the intersection $\bigcap_{j\in T} F_n^j$ is contained in $U$ (and similarly $\bigcap_{j\in T} C_{n}^j\subseteq U$). Let $H$ be the affine hull of $U$, which is a $u$-dimensional affine space. %(which we identify with $\R^{u}$ in the sequel).
   Since intrinsic volumes are independent of the dimension of the
   ambient space, it is enough to study the intersection of the sets
   $F_n^j\cap H$, $j\in T$ (or $C_{n}^j\cap H$, respectively) in the
   space $H$. The sets $F^j\cap H$ can be modeled by fractal
   percolations on $u$-dimensional cubes. Let $K^{(j)}$, $j\in T${,}
   be independent fractal percolations on $[0,1]^u$ with the same
   parameters $p$ and $M$ as $F$. Denote by $\widetilde{K}^{(j)}$ the
   random set which equals $K^{(j)}$ with probability $p$ and is empty
   otherwise. Then we have, for each $n\in\N$,
   $F_n^j\cap H=\psi(\widetilde{K}_n^{(j)})$, $j\in T$, in distribution, where $\psi:H\to\R^{u}$ is one of the similarities (with factor $1/r$) mapping $U$ to $[0,1]^{u}$. (In fact, it is possible to couple $F^j$ and $\widetilde{K}^{(j)}$ in such a way that this distributional relation becomes an almost sure one. But we do not need this here for our argument.) In particular, it follows that
   \begin{align*}
      C_k^\var(\bigcap_{j\in T} F_n^j)&=C_k^\var(\bigcap_{j\in T} \psi(\widetilde{K}_n^{(j)}))\leq r^{-k} C_k^\var(\bigcap_{j\in T} K_n^{(j)}),
   \end{align*}
   where the equality holds in distribution and the inequality almost surely. For the latter note that, for each realization of the random sets $K^{(j)}$, $j\in T$, either $\widetilde{K}_n^{(j)}=\emptyset$, for some $j\in T$, in which case the inequality is trivial, since the total variation measure is non-negative, or $\widetilde{K}_n^{(j)}={K}_n^{(j)}$ for all $j\in T$, in which case it becomes an equality.
   (Analogously, we have $C_{n}^j\cap H=\psi(\widetilde{D}_n^{(j)})$, $j\in T$, in distribution, where $\widetilde{D}_n^{(j)}:=\overline{[0,1]^u\setminus \widetilde{K}_n^{(j)}}$ is the random set which equals $D_n^{(j)}:=\overline{[0,1]^u\setminus K_n^{(j)}}$ with probability $p$ and $[0,1]^u$ with probability $1-p$.)
}

{Now we are in the position to apply Lemma~\ref{lem:intersect-num}. For each realization of the $K^{(j)}$, $j\in T$, the sets $K^{(j)}_n$ are collections of ($u$-dimensional) level-$n$ cubes as required. Fixing some index $j'\in T$, we denote by $Z'_n$ the number of basic cubes of level $n$ contained in $K^{(j')}$ and infer that
\begin{align*}
      C_k^\var(\bigcap_{j\in T} K_n^{(j)})\leq c_{u,k} Z'_n r^{kn},
   \end{align*}
where $c_{u,k}$ is a universal constant independent of the realization, which means that the above estimate holds almost surely. Observe that the random variables $Z'_n$, $n\in\N$ form a Galton-Watson process with offspring distribution $\text{Bin}(M^u,p)$. It is well known that $\E Z'_n=(M^{u}p)^n$. Setting $c:=r^{-k}c_{u,k}$, we conclude that \begin{align*}
      \E C_k^\var(\bigcap_{j\in T} F_n^j)&\leq r^{-k} \E C_k^\var(\bigcap_{j\in T} K_n^{(j)})\leq c M^{(u-k)n}p^n,
   \end{align*}
which implies (recall that $M^D=M^dp$)
\begin{align*}
     \sum_{n=1}^\infty r^{(D-k)n} \E C_k^\var(\bigcap_{j\in T} F_n^j)&\leq c \sum_{n=1}^\infty r^{(d-u)n}= \frac c{1-r^{d-u}}<\infty,
   \end{align*}
   since $u\leq d-1$. This proves the first assertion in Proposition~\ref{prop:Rk-conv} and completes the proof of this statement.
}

{Now let us look at the first assertion of Proposition~\ref{prop:Rk-conv2}. Fix $k\in\{0,\ldots, d-1\}$. An analogous argument as above now applied to the intersections $\bigcap_{j\in T}C^j_n$ yields
\begin{align*}
      C_k^\var(\bigcap_{j\in T} C_{n}^j)&= r^{-k}C_k^\var(\bigcap_{j\in T} \widetilde{D}_n^{(j)}) %\leq r^{-k}\max_{T'\subset T}C_k^\var(\bigcap_{j\in T'} D_n^{(j)})
      \leq c M^{un} r^{kn},
      \end{align*}
where $c$ is as above and $M^{un}$ is an upper bound for the number of basic cubes of level $n$ in the set $\widetilde{D}^{(j')}_n=\overline{[0,1]^u\setminus \widetilde{K}^{(j')}_n}$ (for some fixed index $j'\in T$) even in the case when this set equals $[0,1]^u$. This estimate is not as strong as the one in the previous case. %In fact, if we want the argument to work for all sets $T$, we can only use that $u\leq d-1$
Nevertheless, taking expectations, multiplying by $r^{(D-k)n}$ and
summing over $n$ as above yields the desired conclusion, the finiteness
of the expression $ \sum_{n=1}^\infty r^{n(D-k)} \E
C_k^\var(\bigcap_{j\in T} C_{n}^j)$, for all $p\in(r^{d-u},1]$.  This
proves the first assertion in Proposition~\ref{prop:Rk-conv2} for any
index set $T$ such that $u(=u(T))\leq k$, for the full range of $p$ for
which we claimed it (and for some $T$ even for some more $p$). For $T$
such that $u>k$, however, some $p$ are missing (namely the interval
$(r^{d-k},r^{d-u}]$). But in this case, a better estimate can be
obtained by taking Remark~\ref{rem:bd-estimate} into account. It allows
to replace $M^{un}$ in the above estimate by the number $Y_n'$ of
level-$n$ cubes in $\widetilde{D}_n^{(j')}$ that have a nonempty
intersection with $\partial\big(\bigcap_{j\in T}
\widetilde{D}_n^{(j)}\big)$. For any cube $Q$ counted in $Y'_n$ there is
an index $j\in T$ such that $C\cap\partial \widetilde{D}_n^{(j)}\neq
\emptyset$. This in turn means that either one of the neighboring cubes
of $Q$ (i.e., one with nonempty intersection with $Q$) is contained in
$K_n^{(j)}$, or $Q\cap\partial[0,1]^u\neq \emptyset$. Hence the number
$Y'_n$ can be bounded from above by the number of neighbors of the cubes
in the sets $K^{(j)}_n$, $j\in T$, plus the number of cubes intersecting
$\partial[0,1]^u$. In fact, since the $k$-th curvature measure of the
cube $[0,1]^u$ is concentrated on its $k$-faces (recall $k<u$), among
the cubes $Q$ intersecting $\partial[0,1]^u$ that are not already
counted because of their intersection with some cube in some
$K^{(j)}_n${,} only those need to be counted, which intersect a
$k$-face of $\partial[0,1]^u$. Note that there are at most $M^{kn}$ such
cubes for each $k$-face. Summarizing the argument, we obtain
\begin{align*}
      C_k^\var(\bigcap_{j\in T} C_{n}^j)&= r^{-k}C_k^\var(\bigcap_{j\in T} \widetilde{D}_n^{(j)})\leq c r^{kn}\left(\sum_{j\in T} (3^u-1) |K^{(j)}_n|+ f_k M^{kn}\right),
      \end{align*}
where $f_k$ is the number of $k$-faces of $[0,1]^u$ and $3^u-1$ is  the number of neighboring cubes of the same size of a given $u$-dimensional cube. Taking expectations and noting that $\E|K^{(j)}_n|=(M^u p)^n$ for each $j\in T$, we get
\begin{align*}
    r^{(D-k)n} \E C_k^\var(\bigcap_{j\in T} C_{n}^j)&\leq c_1 r^{Dn}(M^{u}p)^n+ c_2 r^{Dn} M^{kn}=c_1 r^{(d-u)n}+c_2 r^{(D-k)n},
\end{align*}
where $c_1=c \cdot|T| (3^u -1)$ and $c_2=c\cdot f_k$. Summing over $n$, we conclude that
\begin{align*}
     \sum_{n=1}^\infty r^{(D-k)n} \E C_k^\var(\bigcap_{j\in T} C_n^j)&\leq c_1 \sum_{n=1}^\infty r^{(d-u)n} + c_2 \sum_{n=1}^\infty r^{(D-k)n},
   \end{align*}
where the first sum on the right converges (for all $p\in(0,1]$) since $u\leq d-1$ and the second sum converges for all $p$ such that $D>k$ (i.e.\ for $p\in(r^{d-k},1]$).
This shows the convergence of the sum on the left for any $p\in(r^{d-k},1]$ for the case $k<u$ and completes the proof of the first assertion of Proposition~\ref{prop:Rk-conv2}.
}
\end{proof}

%\acks 
\paragraph{\bf Acknowledgements.}
We thank Klaus Mecke and Philipp Sch\"onh\"ofer, for inspiring
discussions and their preliminary work \cite{MeckeSchoenhoefer15,
Schoenhoefer14} which motivated the authors to look at this model more
closely. {We are grateful to Erik Broman and Federico Camia for
helpful comments on the previous results thresholds in fractal
percolation.}
During the work on this project both authors have been members of the
DFG research unit \emph{Geometry and Physics of Spatial Random Systems}
at Karlsruhe. We gratefully acknowledge support from grants  number
HU1874/3-2, and LA965/6-2.
Part of this research was carried out while the second author was
staying at the \emph{Institut Mittag-Leffler} participating in the 2017
research programme
\emph{Fractal Geometry and Dynamics}. He would like to thank the staff
as well as the participants and organizers for the stimulating
atmosphere and support.

\bibliographystyle{plain} %{apt} %
\bibliography{mp2020-arxiv}

\end{document}